\documentclass[english, final, 11pt]{article}

\usepackage[utf8]{inputenc} 
\usepackage[T1]{fontenc} 
\usepackage{amssymb,amsthm,amsmath} 
\usepackage{babel} 
\usepackage{svn} 
\usepackage{ifdraft} 
\usepackage[notref,notcite]{showkeys} 
\usepackage{graphicx} 
\usepackage[usenames]{color}

\usepackage[unicode=true, bookmarks=false, breaklinks=false,pdfborder={0 0 1},backref=false,colorlinks=false]{hyperref}

\newcommand{\ev}[1]{\mathbb{E}{#1}}

\newcommand{\evx}[1]{\mathbb{E}_x{#1}}

\newcommand{\pr}[1]{\mathbb{P}\rbr{#1}}

\newcommand{\norm}[3]{\Vert #1 \Vert_{ #2 } ^{ #3 }} 
\newcommand{\rbr}[1]{\left( #1 \right)} 
\newcommand{\sbr}[1]{\left[ #1 \right]} 
\newcommand{\cbr}[1]{\left\{ #1 \right\}}

\newcommand{\ddp}[2]{\left\langle #1, #2 \right\rangle} 
\newcommand{\intr}{\int_{\mathbb{R}^d}} 
\newcommand{\inti}{\int_{0}^{+\infty}} 
\newcommand{\intc}[1]{\int_{0}^{#1}}

\newcommand{\Rd}{\mathbb{R}^d} 
 
\newcommand{\R}{\mathbb{R}} 
\newcommand{\T}[1]{\mathcal{T}_{#1}} 
 \DeclareMathOperator{\var}{Var} \DeclareMathOperator{\cov}{Cov} \DeclareMathOperator{\grad}{grad}

\definecolor{czerwony}{rgb}{1,0.3,0.3} 
\definecolor{zielony}{rgb}{0.1,0.8,0.4} 
\definecolor{blue}{rgb}{0.1,0.1,0.9}

\newcommand{\dd}[1]{\textnormal{d}#1}

\newcommand{\SP}{\mathcal{S}'(\Rd)} 
\newcommand{\SD}{\mathcal{S}(\Rd)}

\newcommand{\ceq}{\eqsim} 
\newcommand{\cleq}{\lesssim}

\newtheorem{theorem}{Theorem}[section] 
\newtheorem{lemma}[theorem]{Lemma} 
\newtheorem{fact}[theorem]{Fact}

 \theoremstyle{remark} 
\newtheorem{rem}[theorem]{Remark}

\usepackage[left=2cm,top=2cm,right=2cm,,bottom=3cm,nohead]{geometry}

\newcommand{\eq}{\varphi} 
 
  \DeclareMathOperator{\Lip}{Lip} 
 
\newcommand{\pspace}[1]{\mathcal{P}^{#1}}

\title{CLT for Ornstein-Uhlenbeck branching particle system} 
\author{\textbf{Radosław Adamczak\footnote{radamcz@mimuw.edu.pl. Research was partially supported by the MNiSW grant N N201 397437 and by the Foundation for Polish Science.} \;and Piotr Miłoś\footnote{pmilos@mimuw.edu.pl (corresponding author). Research was partially supported by the MNiSW grant N N201 397537.}}\\Faculty of Mathematics, Informatics and Mechanics, University of Warsaw\\ul. Banacha 2, Warsaw, Poland} 
\begin{document}

\maketitle 
\begin{abstract}
	In this paper we consider a branching particle system consisting of particles moving according to the Ornstein-Uhlenbeck process in $\Rd$ and undergoing a binary, supercritical branching with a constant rate $\lambda>0$. This system is known to fulfil a law of large numbers (under exponential scaling). In the paper we prove the corresponding central limit theorem. The limit and the CLT normalisation fall into three qualitatively different classes. In, what we call, the small branching rate case the situation resembles the classical one. The weak limit is Gaussian and normalisation is the square root of the size of the system. In the critical case the limit is still Gaussian, however the normalisation requires an additional term. Finally, when branching has large rate the situation is completely different. The limit is no longer Gaussian, the normalisation is substantially larger than the classical one and the convergence holds in probability. \\
	We prove also that the spatial fluctuations are asymptotically independent of the fluctuations of the total number of particles (which is a Galton-Watson process). ~\\
	~\\
	MSC: primary 60F05; 60J80 secondary 60G20 \\
	Keywords: Supercritical branching particle systems, Central limit theorem, Galton-Watson process.
\end{abstract}
\section{Introduction} 

\label{sec:introduction} We consider a branching particle system $\cbr{X_t}_{t\geq 0}$ as follows. The system starts off at time $t=0$ from a single particle located at $x\in \Rd$. The particle moves according to the Ornstein-Uhlenbeck process in $\Rd$ and branches after exponential time with parameter $\lambda>0$. The branching is supercritical and given by the generating function
\[ F(s) := p s^2 + (1-p), \quad p>\frac{1}{2}. \]
The offspring particles follow the same dynamics. This system will be referred to as the OU branching process. Formally, the system is identified with the empirical process, i.e. $X$ is a measure-valued process such that for a Borel set $A$, $X_t(A)$ is the (random) number of particles at time $t$ in $A$. We recall that the Ornstein-Uhlenbeck process is a time homogenous Markov process with the infinitesimal operator 
\begin{equation}
	L := \frac{1}{2} \sigma^2 \Delta - \mu x\circ \nabla. \label{eq:ouOperator} 
\end{equation}
where $\sigma,\mu>0$ and $\circ$ denotes the standard scalar product. The Ornstein-Uhlenbeck process has a unique equilibrium measure $\eq$ (to be described later).

Systems of this type may be regarded as consisting of two components, namely the genealogy part and diffusion part (for this reason they are sometimes called ``branching diffusions''). The genealogy part, being the celebrated Galton-Watson process is well-studied. In our paper, the expected number of the progeny of a particle is strictly greater than $1$, therefore the system is supercritical. The expected total number of particles grows exponentially at the rate 
\begin{equation}
	\lambda_p :=(2p-1)\lambda. \label{eq:growth-rate} 
\end{equation}
After a long time the positions of two ``randomly picked'' particles are ``almost independent'' random variables, which suggests the following law of large numbers 
\begin{equation}
	|X_t|^{-1} \ddp{X_t}{f} \rightarrow \ddp{\eq}{f}1_{Ext^c},\quad a.s. \label{eq:basicLLN} 
\end{equation}
where $Ext^c$ is the event that the system does not become extinct, $|X_t|$ denotes the number of particles at time $t$ and $f$ is a bounded, continuous function. This is indeed the case as follows from \cite{Harris:2008wd}. In Theorem \ref{thm:LLN} we obtain (\ref{eq:basicLLN}) for a slightly more general class of functions. This is, however, only a preparatory step for our main goal which is the corresponding central limit theorem. The second order behaviour depends qualitatively on the sign of $\lambda_p - 2\mu$. Roughly speaking this condition reflects the interplay of two antagonistic forces, the growth which is local and makes the system more coarse and the movement which tends to smooth the system (higher $\mu$ implies ``a stronger attraction of particles towards $0$''). Now we are going to describe the behaviour of the spatial fluctuations: 
\begin{equation}
	F_t^{-1}\rbr{\ddp{X_t}{f} - |X_t|\ddp{\eq}{f}}, \label{eq:spatialFluct} 
\end{equation}
where $F_t$ is some not necessarily deterministic norming. We will describe the situation on the set of non-extinction $Ext^c$. 
\begin{description}
	\item[Small branching rate: $\lambda_p<2\mu$.] Our main result is contained in Theorem \ref{thm:clt1}. In this the case ``the movement part prevails'' and the result resembles the standard CLT. The normalisation is given by $F_t = |X_t|^{1/2}$ (which is of order $e^{(\lambda_p/2) t}$). Moreover, we obtain the limit which is Gaussian (though its variance is given by a complicated formula) and does not depend on the starting position $x$. Let us also note that a random normalisation is quite natural as it ``filters'' out the fluctuation of the total number of particles. 
	\item[Critical branching rate: $\lambda_p=2\mu$.] Our main result is contained in Theorem \ref{thm:cltCritical}. In this case ``the branching prevails''. The behaviour of the fluctuations slightly diverge from the classical setting. The normalisation is bigger: $F_t = t^{1/2}|X_t|^{1/2}$. The limit still does not depend on the starting condition and is Gaussian but its variance depends on the derivatives of $f$. To explain this we notice that the branching is so fast that the fluctuations are are not smoothen by the motion and become essentially local. In consequence they give rise to a spatial white noise and larger normalisation is required. 
	\item[Large branching rate: $\lambda_p>2\mu$.] Our main result is contained in Theorem \ref{thm:clt2}. In this case not only does the branching ``prevail'' but also ``the motion badly fails to make any smoothing''. The normalisation is even bigger: $F_t = e^{(\lambda_p-\mu) t}$ and we have $\lambda_p-\mu>\lambda_p/2$. The limit is no longer Gaussian, it is given by $\ddp{ f}{\grad \varphi}\circ J$ (where $J$ is the limit of a certain martingale). What is perhaps surprising, the limit holds in probability. The first term, $\ddp{f}{\grad \varphi}$, means that alike the critical situation the branching is fast enough to produce some sort of a white noise. Even more, it is so fast that the limit depends on the starting condition and in fact, up to some extent, the system ``remembers its whole evolution'', which is encoded in $J$. 
\end{description}
In either case we prove also that the spatial fluctuations become independent of fluctuations of the total number of particles as time increases.

To our best knowledge so far there have been no CLT-type results of a similar flavour in the field of branching diffusions and the research effort was concentrated instead on proving laws of large numbers for more and more general branching systems. The only CLT result we are aware of is contained in \cite[Proposition 6.4]{Bansaye:2009cl}. Their setting is somewhat different as they consider ``running fluctuations'' of the form $F_{t,T}^{-1} \rbr{ \ddp{X_{t+T}}{f} - \ddp{X_{T}}{f_{t,T}}}$ as $T\rightarrow +\infty$, where $F_{T,t}$ is a normalisation and $f_{T,t}$ is a certain transformation of $f$. Besides many advantages their approach fails to capture the emergence of three qualitatively different cases described above. More detailed discussion is contained in Remark \ref{rem:comparison}. It is also noteworthy that our results open a possibility of further research. We list a few most promising possibilities in Remark \ref{rem:others1} and Remark \ref{rem:others2}.

At this moment we would like to announce our parallel paper \cite{Adamczak:2011fk}. In \cite{Adamczak:2011fk} we consider the $U$-statistics of the OU branching system, namely expressions of the form 
\begin{equation}
	U^n_t(f) := \sum_{ \substack{i_1,i_2, \ldots i_n=1 \\
	i_k \neq i_j\text{ if } k\neq j}}^{|X_t|} f(X_t(i_1), X_t(i_2), \ldots, X_t(i_n)). 
\end{equation}
We obtain both the law of large numbers (i.e. an analogue of \eqref{eq:basicLLN}) and CLTs, which also fall into three categories corresponding to the cases described above. Moreover, at this point we also advertise a forthcoming work of the second named author \cite{Mios:2011fk}, which is devoted to studies of the CLT for superprocesses based on the Ornstein-Uhlenbeck process. Qualitatively the results of \cite{Mios:2011fk} are the the same as the ones presented in this paper.

Our proofs utilise a mixture of techniques used for the branching particle systems (e.g. the Laplace transform and the log-Laplace equation, coupling and decoupling). Although the proof schemes loosely resemble known techniques (e.g. are similar to proofs in \cite[Section 1.13]{Athreya:2004xr}) they required many improvements. In our proofs we used also the fact that the Ornstein-Uhlenbeck process has a particularly explicit and traceable structure. Improving this part seems to be an interesting research problem.\\
The studies of branching models of various types have a long history, we refer the reader to \cite{Dynkin:1994aa,Etheridge:2000fe,Athreya:2004xr,Dawson:1993lr} (the list is by no means exhaustive). It has been known for a long time that the behaviour of branching systems differs qualitatively for (sub)critical and supercritical cases. The former become extinct almost surely, their limit properties are studied after conditioning on non-extinction event (e.g. Yoglom type theorems) or as a part of larger infinite structures (e.g. Galton-Watson forests, random snakes, continuum trees e.g. \cite{Gall:1999pi}). The latter grow exponentially fast (on the set of non-extinction), which makes them possible to be studied using laws of large numbers, starting with the celebrated Kesten-Stigum theorem (\cite[Theorem I.10.1]{Athreya:2004xr}). Such theorems were also proved for branching particle systems, they go back to \cite{Asmussen:1976ph,Asmussen:1976pi} and more recently \cite{Englander:2007im} (which was the main inspiration for our paper). It presents a law of large number for a large class of supercritical branching diffusions and admits unbounded space-dependent branching intensity. Also we would like to mention again paper \cite{Bansaye:2009cl} in which systems with non-local branching (i.e. particles may jump upon an event of branching) were studied. The article \cite{Bansaye:2009cl} presents a law of large numbers as well as a central limit theorem, though in a different spirit than ours (see comparison in Remark \ref{rem:comparison}). Due to its excessive introduction \cite{Bansaye:2009cl} is also an excellent resource of biological motivations for study of branching diffusions.\\
The article is organised as follows. The next section presents notation and basic facts required further. Section \ref{sec:results} is devoted to presentation of results. Proofs are deferred to Section \ref{sec:proofs} and the Appendix.

\section{Definitions and notation } 

\subsection{Notation} \label{sec:definitions_and_notation} For a branching system $\cbr{X_t}_{t\geq 0}$, we denote by $|X_t|$ the number of particles at time $t$, and by $X_t(i)$ - the position of the $i$-th (in a certain ordering) particle at time $t$. Typically, we use $\evx{}$ or $\mathbb{P}_x$ to stress the fact that we calculate the expectation for the system starting from a particle located at $x$. Sometimes we use also $\ev{}$ and $\mathbb{P}$ when this location is not relevant (e.g. if we calculate the number of particles in the system). We will refer to system starting from a single particle at time $t=0$ located at $x\in \Rd$ shortly, as the\:OU\:branching system starting from $x\in \Rd$.\\
For a function $f \colon \R^d \to \R$, we will denote 
\begin{displaymath}
	\langle X_t, f \rangle = \sum_{i=1}^{|X_t|}f(X_t(i)). 
\end{displaymath}
By $\rightarrow^d $ we denote the convergence in law. We use $\cleq, \ceq$ to denote the situation when an equality or inequality holds with a constant $c>0$, which is irrelevant for calculations. E.g. $f(x)\ceq g(x)$ means that there exists a constant $c>0$ such that $f(x) = c g(x)$.\\
Let $x \circ y = \sum_{i=1}^d x_i y_i$ denote the standard scalar product of $x,y\in \Rd$. Moreover, $\norm{x}{}{} = \sqrt{x\circ x}$ is the standard Euclidean norm in $\R^d$.

We use also $\ddp{f}{\mu}:=\int_{\Rd} f(x) \mu{(\dd{x}) }$. 

In the paper we will use the space
\[ \pspace{}=\pspace{}(\R^{d}):=\cbr{f:\R^{d}\mapsto\R:f\,\text{ is continuous and } \exists_{n}\text{ such that }|f(x)|/\norm x{}n\rightarrow 0\text{ as }\norm x{}{}\rightarrow+\infty},
\]
that is the space of continuous functions which grow at most polynomially. \\
Given a function $f\in \pspace{}(\R^d)$ we will implicitly understand its derivatives (e.g. $\frac{
\partial f}{
\partial x_i}$) in the space of the tempered distribution (see e.g. \cite[p. 173]{Rudin:1973fk}).

\paragraph{Basic facts on the Galton-Watson process} The number of particles $\cbr{|X_t|}_{t\geq 0}$ is the celebrated Galton-Watson process. We present basic properties of this process used in the paper. The main reference in this section is \cite{Athreya:2004xr}. We already introduced the growth rate \eqref{eq:growth-rate} (e.g. \cite[Section 1.6]{Athreya:2004xr}). The process becomes extinct with the probability $p_e := \frac{1-p}{p}.$ (see \cite[Theorem I.5.1]{Athreya:2004xr}). We will denote the extinction and non-extinction events by $Ext$ and $Ext^c$ respectively. The process $$V_t := e^{-\lambda_p t} |X_t|$$ is a positive martingale. Therefore it converges (see also \cite[Theorem 1.6.1]{Athreya:2004xr}) 
\begin{equation}
	V_t \rightarrow V_\infty, \quad a.s. \:\:\text{ as }t\rightarrow +\infty. \label{eq:defV} 
\end{equation}
\begin{fact}
	\label{fact:aaa} We have $\cbr{V_\infty = 0} = Ext$ and conditioned on non-extinction $V_\infty$ has the exponential distribution with parameter $\frac{2p-1}{p}$. We have $\ev{}(V_\infty) =1$ and $ \var(V_\infty) = \frac{1}{2p-1}$. $\ev{}e^{-4\lambda_p t} |X_t|^4$ is uniformly bounded, i.e. there exists $C>0$ such that for any $t\geq 0$ we have $\evx{}e^{-4\lambda_p t}|X_t|^4 \leq C$. Moreover, all moments are finite, i.e. for any $n\in \mathbb{N}$ and $t\geq 0$ we have $\ev{} |X_t|^n < +\infty$. 
\end{fact}
The proof is deferred to the Appendix. We will denote the variable $V_\infty$ conditioned on non-extinction by $W$. The density of the invariant measure of the Ornstein-Uhlenbeck process defined by \eqref{eq:ouOperator} is given by 
\begin{equation}
	\eq(x) := \rbr{\frac{\mu}{\pi \sigma^2}}^{d/2}\exp \rbr{-\frac{\mu}{\sigma^2} \norm{x}{}{2}}. \label{eq:equilibrium} 
\end{equation}

\section{Results} 

\label{sec:results} This section is devoted to the presentation of our results. The proofs are deferred to Section \ref{sec:proofs}. Our first aim is to present a central limit theorem corresponding to the following law of large numbers (closely related to \cite[Theorem 6]{Englander:2007im} or \cite[Theorem 4.2]{Bansaye:2009cl}) 
\begin{theorem}
	\label{thm:LLN} Let $\cbr{X_t}_{t\geq 0}$ be the OU branching system starting from $x\in\Rd$. Let us assume that $f \in \pspace{}(\Rd)$. Then
	\[ \lim_{t \rightarrow +\infty} e^{-\lambda_p t}\ddp{X_t}{f} = \ddp{f}{\eq} V_\infty \:\: \text{ in probability}, \]
	or equivalently on the set of non-extinction, $Ext^c$, we have 
	\begin{equation}
		\lim_{t \rightarrow +\infty} |X_t|^{-1} \ddp{X_t}{f} = \ddp{f}{\eq} \:\: \text{ in probability}. \label{eq:LLN-simple} 
	\end{equation}
	Moreover, if $f$ is bounded then the almost sure convergence holds. 
\end{theorem}
\subsection{Small branching rate: $\lambda_p < 2\mu$} We denote $\tilde{f}(x):=f(x)- \ddp{f}{\varphi}$ and 
\begin{equation}
	\sigma_f^2 := \ddp{\eq}{\tilde{f}^2} + 2\lambda p \inti \ddp{\eq}{\rbr{ e^{(\lambda_p/2) s}\T{s} \tilde{f}}^2} \dd{s}. \label{eq:sigmaf} 
\end{equation}
The above expression becomes perhaps less cryptic when expressed in the base of the Hermite polynomials. Let $\cbr{H_i}_{i\geq0}$ be the probabilistic Hermite polynomials (see. e.g. \cite[page 5]{Harris:2000fy}) and let $f_{i_1,i_2,\ldots, i_d} := \intr \tilde{f}(x) \prod_{j=1}^d H_{i_j}(x_j) \eq(x) \dd{x}$. Obviously, $f_{0,0,\ldots, 0} =0$. Moreover, we have 
\begin{equation}
	\sigma_f^2 := \sum_{i_1=0, i_2 =0,\ldots, i_d=0}^{+\infty} f_{i_1, i_2, \ldots, i_d}^2\rbr{1+ \frac{ 2\lambda p}{2(i_1+i_2 +\ldots +i_d)\mu -\lambda_p}}. \label{eq:sigmafSubcriticalHermite} 
\end{equation}
We will comment on this formula in Remark \ref{rem:martingales}.

Let us also recall \eqref{eq:defV} and that $W$ is $V_\infty$ conditioned on $Ext^c$. The main result of this section is 
\begin{theorem}
	\label{thm:clt1} Let $\cbr{X_t}_{t\geq 0}$ be the OU branching system starting from $x\in \Rd$. Let us assume $\lambda_p<2\mu$ and $f\in \pspace{}(\Rd)$. Then $\sigma_f^2 <+\infty$ and conditionally on the set of non-extinction $Ext^c$ there is the convergence
	\[ \rbr{e^{-\lambda_p t}|X_t| , \frac{|X_t| - e^{t \lambda_p} V_\infty}{\sqrt{|X_t|}}, \frac{\ddp{X_t}{f} - |X_t|\ddp{f}{\eq} }{ \sqrt{|X_t|} } } \rightarrow^d (W, G_1, G_2), \]
	where $G_1\sim \mathcal{N}(0, 1/(2p-1)), G_2\sim \mathcal{N}(0,\sigma_f^2)$ and $W, G_1, G_2$ are independent random variables. 
\end{theorem}
\begin{rem}
	\label{rem:jakis} As we already mentioned in the Introduction this is the most classical case. One can check that the theorem is still valid when the first particle is distributed according to $\eq$. The third term in the theorem can be written as: $|X_t|^{-1/2}\sum_{i=1}^{|X_t|} (f(X_t(i)) - \ddp{f}{\eq})$, moreover (by our special assumption about the starting condition) each $X_t(i)\sim \eq$. If these random variables were independent then the above sum would converge to $\mathcal{N}(0,\tilde{\sigma}_f^2)$, where $\tilde{\sigma}_f^2 = \ddp{\eq}{\tilde{f}^2}$ (the random number of elements in the sum is only a minor obstacle). Therefore, the additional integral term in \eqref{eq:sigmaf} reflects the dependence between $X_t(i)$'s. 
\end{rem}
\begin{rem}
	An important feature of our result is the factorisation of the fluctuations of the total mass process $\cbr{|X_t|}_t$ and the spatial fluctuations process i.e. $\cbr{\ddp{X_t}{f}}_t$. Using this fact we can easily prove a central limit theorem corresponding to \eqref{eq:LLN-simple} with deterministic normalisation. On the set of non-extinction $Ext^c$ we have
	\[ e^{-\lambda_p t/2}\rbr{\ddp{X_t}{f} - |X_t|\ddp{f}{\eq}} \rightarrow^d G_2 \sqrt{W}, \]
	where $W,G_2$ are the same as in Theorem \ref{thm:clt1}. 
\end{rem}
\begin{rem}
	The convergence of the spatial fluctuations can also be regarded as convergence of random fields. It is technically convenient to embed the space of point measures into the space of tempered distributions $\SP$ (i.e. the dual of the space of rapidly decreasing functions $\SD$). The L\'evy continuity theorem holds in nuclear spaces (e.g. \cite[Theorem B.]{Gorostiza:1994uq}), hence the following $\SP$-valued random variable: 
	\begin{equation}
		M_t:=\frac{X_t - |X_t| \eq(x) \dd{x} }{|X_t|^{1/2}}, \label{eq:randomMeasure} 
	\end{equation}
	converges to a Gaussian random field $M$ with covariance structure given by 
	\begin{displaymath}
		\cov(\ddp{M}{f_1}, \ddp{M}{f_2}) = \ddp{\eq}{\tilde{f}_1 \tilde{f}_2 } + 2\lambda p \inti \ddp{\eq}{\rbr{ e^{(\lambda_p/2) s}\T{s} \tilde{f}_1}\rbr{ e^{(\lambda_p/2) s}\T{s} \tilde{f}_2}} \dd{s}, 
	\end{displaymath}
	where $\tilde{f}_i(x) = f_i(x) - \ddp{\eq}{f_i}$. 
\end{rem}

\subsection{Critical branching rate: $\lambda_p = 2\mu$} We denote 
\begin{equation}
	\sigma_f^2 := \frac{\lambda p\sigma^2}{\mu} \sum_{i=1}^{d} \ddp{\frac{
	\partial f}{
	\partial x_i} }{\eq}^2. \label{eq:sigmafCritical} 
\end{equation}
Note that the same symbol $\sigma_f^2$ has already been used to denote the asymptotic variance in the small branching case. However, since these cases will always be treated separately, this should not lead to ambiguity. Analogously as in \eqref{eq:sigmafSubcriticalHermite} we can expresses $\sigma_f^2$ nicely using the Hermite expansion: 
\begin{equation}
	\sigma_f^2 = \frac{4\lambda p\mu}{\sigma^2} \rbr{f_{1,0,\ldots,0}^2 +f_{0,1,\ldots,0}^2 + f_{0,0,\ldots,1}^2}. \label{eq:sigmafCriticalHermite} 
\end{equation}
Let us also recall \eqref{eq:defV} and that $W$ is $V_\infty$ conditioned on $Ext^c$. The main result of this section is 
\begin{theorem}
	\label{thm:cltCritical} Let $\cbr{X_t}_{t\geq 0}$ be the OU branching system starting from $x\in \Rd$. Let us assume that $\lambda_p=2\mu$ and $f\in \pspace{}(\Rd)$. Then $\sigma_f^2 <+\infty$ and conditionally on the set of non-extinction $Ext^c$ there is the convergence
	\[ \rbr{e^{-\lambda_p t}|X_t| , \frac{|X_t| - e^{t \lambda_p} V_\infty}{\sqrt{|X_t|}}, \frac{\ddp{X_t}{f} - |X_t|\ddp{f}{\eq} }{ t^{1/2} \sqrt{|X_t|} } } \rightarrow^d (W, G_1, G_2), \]
	where $G_1\sim \mathcal{N}(0, 1/(2p-1)), G_2\sim \mathcal{N}(0,\sigma_f^2)$ and $W, G_1, G_2$ are independent random variables. 
\end{theorem}
\begin{rem}
	We continue discussion from Remark \ref{rem:jakis}. This time the theorem is less classical as the normalisation is larger. We interpret this as the fact that $X_t(i)$'s become more dependent as $\lambda_p$ increases relatively to $\mu$. In other words the branching is so fast that any particle has many relatives which are still close to it. This also explains, at least on an intuitive level, the appearance of derivative in \eqref{eq:sigmafCritical}. One can think that the limit, in a sense of random fields, is a Gaussian white noise. 
\end{rem}

\subsection{Large branching rate: $\lambda_p > 2\mu$} We introduce the process 
\begin{equation}
	H_t := e^{(-\lambda_p + \mu)t} \sum_{i=1}^{|X_t|} X_t(i),\quad t\geq 0. \label{eq:martingale} 
\end{equation}
We have 
\begin{fact}
	\label{fact:martingaleConvergence} Let $\lambda_p>2\mu$. Then $H_t$ is a martingale with respect to the filtration of the OU branching system. We have $\sup_t \evx{H_t^2} <+\infty$, therefore there exists $$H_\infty := \lim_{t\rightarrow +\infty} H_t,$$ where the convergence holds a.s. and in $L^2$. When the OU branching system starts from $0$, then the martingales $V_t$ and $H_t$ are orthogonal. 
\end{fact}
The distribution of $H_\infty$ depends on the starting conditions. 
\begin{fact}
	\label{fact:law} Let $\cbr{X_t}_{t\geq 0}$ and $\{\tilde{X}_t\}_{t\geq 0}$ be two OU branching systems, the first one starting from $0$ and the second one from $x$. Let us denote the limit of the corresponding martingales by $H_\infty,\tilde{H}_\infty$ respectively. Then
	\[ \tilde{H}_\infty =^d H_\infty +x V_\infty , \]
	where $V_\infty$ is given by \eqref{eq:defV} for the system $X$. 
\end{fact}

Let us denote by $J$ the random variable $H_\infty$ conditioned on $Ext^c$.

\begin{theorem}
	\label{thm:clt2} Let $\cbr{X_t}_{t\geq 0}$ be the OU branching system starting from $x\in \Rd$. Let us assume that $\lambda_p>2\mu$ and $f\in \pspace{}(\Rd)$. Then conditionally on the set of non-extinction $Ext^c$ there is the convergence 
	\begin{equation}
		\rbr{e^{-\lambda_p t}|X_t| , \frac{|X_t| - e^{t \lambda_p} V_\infty}{\sqrt{|X_t|}}, \frac{\ddp{X_t}{f} - |X_t|\ddp{f}{\eq} }{ \exp\rbr{(\lambda-\mu)t} } } \rightarrow^d (W, G, \ddp{\grad f}{\eq} \circ J), \label{eq:triple} 
	\end{equation}
	where $G\sim \mathcal{N}(0, 1/(2p-1))$ and $(W,J), G$ are independent. Moreover 
	\begin{equation}
		\rbr{e^{-\lambda_p t}|X_t|, \frac{\ddp{X_t}{f} - |X_t|\ddp{f}{\eq} }{ \exp\rbr{(\lambda-\mu)t} } } \rightarrow (V_\infty,\ddp{\grad f}{\eq} \circ H_\infty), \quad \text{in probability.} \label{eq:inProbablity} 
	\end{equation}
\end{theorem}
\begin{rem}
	As we noted in the Introduction this case hardly resembles the classical CLT. The convergence in probability is perhaps its most unexpected feature. Unfortunately, we cannot present any satisfying intuitive explanation. This phenomenon seems to be closely related to the fact that the branching is so fast that the system is ``not able to forget'' the starting condition and in fact, up to some degree, it ``remembers'' its whole evolution (encoded in martingale $H_t$). Analogously to the critical case the limit, treated as a random field, it is some sort of white noise (not Gaussian). 
\end{rem}
\begin{rem}
	We were not able to derive any explicit formula for the law of $H_\infty$. However we calculated some of its moments. As the formulas become lengthy, we assume that $p=1$ and present only:
	\[ \ev{}_0 (H_\infty)^2 = 2\gamma,\quad \ev{}_0 (H_\infty)^4=\frac{96 \gamma ^2 \left(16+39 \gamma +30 \gamma ^2+8 \gamma ^3\right)}{9+27 \gamma +26 \gamma ^2+8 \gamma ^3}, \]
	{\tiny
	\[ \ev{}_0 (H_\infty)^6 = \frac{1440 \gamma ^3 \left(36847+285675 \gamma +948012 \gamma ^2+1760420 \gamma ^3+2005408 \gamma ^4+1441120 \gamma ^5+642112 \gamma ^6+163584 \gamma ^7+18432 \gamma ^8\right)}{(1+\gamma )^2 (3+2 \gamma ) (5+4 \gamma ) (5+6 \gamma ) (5+8 \gamma ) \left(6+17 \gamma +12 \gamma ^2\right)}, \]
	} where $\gamma = ({\lambda_p}/{\mu}-2)^{-1}$. The even moments are $0$ as the distribution is symmetric. One can now check that $H_\infty$ is never Gaussian. Moreover, $V_\infty$ and $H_\infty$ are not independent (even though uncorrelated!). Their dependence is not trivial. By similar calculations of the fourth moment of $x V_\infty + H_\infty$ we also checked that $H_\infty$ is not of the form $V_\infty G$, where $G$ is some random variable (not necessarily normal) independent of $V_\infty$. 
\end{rem}
\begin{rem}
	We suspect that the convergence in \eqref{eq:inProbablity} is in fact almost sure. 
\end{rem}

\subsubsection*{General remarks} Now we will present general remarks common for all cases.
\begin{rem}
	\label{rem:martingales} The forms of \eqref{eq:sigmafSubcriticalHermite} and \eqref{eq:sigmafCriticalHermite} as well as the one of \eqref{eq:martingale} (which is nothing else than $e^{-(\lambda_p -\mu)t} \ddp{X_t}{H_1}$) seem to stem from some hidden underlying structure. This indeed is the case which, for notation reasons, we will explain for $d=1$. Formally, one can write
	\[ \ddp{X_t}{f} = \sum_{i=0}^{+\infty} f_i e^{ (\lambda_p - i \mu) t} M^i_t, \]
	where $f_i$ is the expansion of $f$ in the Hermite base and for $i\geq0$ processes $\cbr{M^i_t}_{t\geq 0}$ are given by
	\[ M^i_t = e^{ -(\lambda_p - i \mu) t} \ddp{X_t}{H_i}. \]
	One easily checks that $M^i$ are martingales. Understanding this expansion and its relation to our results will be subject of further research. An analogous expansion for a slightly different model was presented in \cite{Harris:2000fy}. 
\end{rem}
\begin{rem}
	\label{rem:comparison} A result close to ours was presented in \cite[Section 5.1]{Bansaye:2009cl}, where the authors consider a model $\cbr{Y_t}_{t\geq 0}$ in which particles move according to a diffusion on the real line with the generator $Lf(x) = b(x) f'(x) +\frac{\sigma^2(x)}{2} f''(x)$. Upon the event of branching the position of mother (interpreted in \cite{Bansaye:2009cl} as its value) is split between its progeny according to a certain random law. Further, they define, a family of measure-valued stochastic process $\cbr{\eta^T_t}_{t\geq 0}$ by 
	\begin{equation}
		\ddp{\eta_t^T}{f} := e^{(\lambda_p (t+T))/2} \rbr{\frac{ \ddp{X_{t+T}}{f} }{e^{\lambda p (t+T)}} - \frac{\ddp{X_T}{\T{t} f}}{e^{\lambda_p T}}}, \label{eq:tmp17} 
	\end{equation}
	where $\T{t}f = e^{-\lambda_p t} \ev{}_x \ddp{f}{X_t}$. In a carefully chosen functional space (we skip the details for the sake of brevity) $\eta^T \rightarrow^d \eta$, where $\eta$ is a solution of the following stochastic equation
	\[ \ddp{\eta_t}{f} = \intc{t}\intr \rbr{L f(x) + J f(x)} \eta_s(\dd{x} ) \dd{s} + \sqrt{W} \mathcal{W}_t (f), \]
	where $W$ is an analogue of our $V_\infty$, $\mathcal{W}$ is a certain Gaussian martingale with values in a functional space and $J$ is a certain operator related to the branching. As we already mentioned in the Introduction the expression \eqref{eq:tmp17} represents ``running fluctuations'' i.e. study the differences between two time points at a finite distance $t$ while our approach is more like studying difference between time $t$ and infinity. Consequently, it does not seem to be a direct relation between our result and that of \cite{Bansaye:2009cl}. This is also the reason why their result does not change when the branching intensity is being altered. 
\end{rem}
\begin{rem}
	\label{rem:others1} The most important extension of the present work will be to study systems with particles movement given by more general Markov processes (starting with diffusions). It is also very interesting to study the case of systems with non-homogenous branching rate (like the ones in \cite{Englander:2007im}). These questions are harder to answer than our present result. We expect that spectral theory of operators will play a crucial role. One may suspect that formulation of the corresponding results will be much more involved comparing to the Ornstein-Uhlenbeck case. The exponential rate of convergence to the equilibrium measure is essential for a behaviour similar to our slow and critical branching case. This entails that the strange behaviour of the supercritical case may be common for many natural examples, e.g. the branching Brownian motion. In this paper an essential idea was to use an explicit coupling of Ornstein-Uhlenbeck processes. This methods has an advantage of conceptualising and simplifying the proofs. However, it is not amenable to generalisations. The first step to overcome this problem was made in \cite{Mios:2011fk}. Working with a similar model, the author was able to prove CLT results using only analytical properties of the Ornstein-Uhlenbeck semigroup. We believe that such approach will be much easier to extend. We expect that the task will be easiest in the ``small branching rate'' case.
	
To give the reader a glimpse of forthcoming complications let us consider the branching particle system $X$ in $\R^2$ such that particles move according to independent Ornstein-Uhlenbeck processes with different drift parameters $\mu_1<\mu_2$. Let us denote by $\eq_1,\eq_2$ the invariant measures on the first and second coordinate respectively. We fix some $g_1,g_2:\R \mapsto \R $ such that $\ddp{g_1}{\eq_1} =0$ and $\ddp{g_2}{\eq_2} =0$. Further we denote $f_1 := g_1\otimes 1$, $f_2=1\otimes g_2$. We now consider
	\[ \ddp{X_t}{f_1} = \ddp{X^1_t}{g_1},\quad \ddp{X_t}{f_2} = \ddp{X^2_t}{g_2}. \]
	In the above we stick to notation that $X_t^1(i),X_t^2(i)$ denotes the first, respectively second coordinate of the position of the $i$-th particle at time $t$. The situation is ``standard'' when $\lambda_p<2 \mu_1 = 2 \min(\mu_1, \mu_2)$. By Theorem \ref{thm:clt1} in both cases we obtain convergence to a Gaussian limit (with normalisation $|X_t|^{-1/2}$). This is no longer true when $\lambda_p\geq \min(\mu_1, \mu_2)$, using Theorem \ref{thm:cltCritical} and \ref{thm:clt2} one easily checks that the normalisation ``required'' by $\ddp{X_t}{f_1}$ is strictly larger than the one corresponding to $\ddp{X_t}{f_2}$. Even in this simple example the situation becomes ``singular''. 
	
	We expect that the phenomena described above will hold in more general cases. Therefore, in the subcritical case it should be relatively easy to obtain analogs of our results for more general motion processes. Contrarily, the supercritical case may turn out to be much harder and the description of the limit may depend on the motion (or its infinitesimal operator) in a complex way. 
\end{rem}
\begin{rem}
	\label{rem:others2} Now we list other possible extensions of our results. The first has already been obtained, a parallel paper \cite{Adamczak:2011fk} contains the corresponding results for $U$-statistics (as it was described in Introduction). Secondly, corresponding results was also obtained for superprocesses in \cite{Mios:2011fk}. Another natural extension would be to consider more general branching laws. The first problem to consider is a law of large numbers. The statement of the Kesten-Stigum theorem suggest the weakest possible conditions and the Seneta-Heyde theorem suggests further extensions (we refer the reader to \cite{Biggins:1997uq} and references therein). Secondly, we suspect that at least some of our CLT results hold when the branching law has finite variance. The case of infinite variance branching laws is interesting as well, one should expect a stable law in the limit.
	
	Finally, other interesting lines of research are the large deviation principle and functional convergence. 
\end{rem}

\section{Proofs} 

\label{sec:proofs}

\subsection{Ornstein-Uhlenbeck process} 

\label{sub:ornstein_uhlenbeck} The semigroup of the Ornstein-Uhlenbeck process will be denoted by $\T{}$. It can calculated with the following formula 
\begin{equation}
	\T{t} f(x) = (g_t \ast f)(x_t), \quad x_t:= e^{-\mu t} x, \label{eq:OUrep} 
\end{equation}
where
\[ g_t(x) = \rbr{\frac{\mu}{\pi \sigma_t^2}}^{d/2}\exp \cbr{-\frac{\mu}{\sigma_t^2} x^2}, \quad \sigma_t^2 := \sigma^2(1-e^{-2\mu t}). \]
We denote $ou(t) := \sqrt{1- e^{-2\mu t}}$ and let $G\sim \eq$. The semigroup $\T{}$ has the following useful representations 
\begin{equation}
	\T{t} f(x) = \intr f(x_t - y)g_t(y) \dd{y} = \intr f\rbr{x e^{-\mu t} + ou(t) y } \eq(y) \dd{y} = \ev{} f(xe^{-\mu t} + ou(t) G). \label{eq:semi-group} 
\end{equation}

The following coupling will be very useful for further analysis. 
\begin{fact}
	\label{fact:coupling} There exists a probability space and two Ornstein-Uhlenbeck processes $\cbr{\eta_t}_{t\geq 0}$, $\cbr{\gamma_t}_{t\geq 0}$ defined on this space such that $\eta_0 = x$ and $\gamma_0 = 0$ which fulfil
	\[ \eta_t - \gamma_t = x e^{-\mu t}, \text{ a.s.} \]
\end{fact}
\begin{proof}
	Let $\eta$ be the (unique, strong) solution of the following stochastic differential equation
	\[ \dd{\eta}_t = \sigma \dd{\beta}_t - \mu \eta_t \dd{t}, \quad \eta_0 = x, \]
	where $\beta$ is the standard Wiener process. We construct $\gamma$ by using the same $\beta$ i.e.
	\[ \dd{\gamma}_t = \sigma \dd{\beta}_t - \mu \gamma_t \dd{t}, \quad \gamma_0 = 0. \]
	The result follows by the subtraction
	\[ \dd{(\eta_t - \gamma_t)} = -\mu (\eta_t - \gamma_t) \dd{t}, \quad \eta_0 -\gamma_0 = x. \]
\end{proof}

\subsection{Laplace transform and moments} 

\label{sub:laplace_transform} Now we compute the Laplace transform of $X$ and derive moments. We use standard techniques in the branching processes theory (see e.g. \cite{Gorostiza:1991aa}) hence we skip some details. Let $f:\Rd\mapsto \R_+$ be a continuous compactly supported function. We denote
\[ w(x,t,\theta) := \evx{}\exp\rbr{-\ddp{\theta f}{X_t}}. \]
By standard conditioning, renewal arguments and the branching property we obtain
\[ w(x,t,\theta) = e^{-\lambda t} \T{t} e^{-\theta f}(x) + \lambda \intc{t} e^{-\lambda(t-s)} \T{t-s} F(w(\cdot, s, \theta))(x) \dd{s}, \]
where $F$ is the generating function of the branching law. The last expression writes as 
\begin{equation}
	\frac{d}{dt} w(x,t,\theta) = (L-\lambda) w(x,t,\theta) + (\lambda p) w(x,t,\theta)^2 + \lambda (1-p), \:\: w(x,0,\theta) = e^{-\theta f(x)}. \label{eq:laplace-aux} 
\end{equation}
Hence the Laplace transform satisfies 
\begin{equation}
	w(x,t, \theta) = \T{t} e^{-\theta f }(x) + \lambda \intc{t} \T{t-s} \sbr{p w^2(\cdot, s,\theta) - w(\cdot, s,\theta) + (1-p)}(x) \dd{s}. \label{eq:laplace1} 
\end{equation}
We differentiate \eqref{eq:laplace1} with respect to $\theta$. For $k\geq 1$ we have 
\begin{align*}
	w^{(k)}(x,t,\theta) =& (-1)^k \T{t} \sbr{ f^k(\cdot) e^{-\theta f }(\cdot)}(x) \\
	&+ \lambda \intc{t} \T{t-s} \sbr{ p \sum_{l=0}^k \binom{k}{l} w^{(l)}(\cdot, s,\theta) w^{(k-l)}(\cdot, s,\theta) - w^{(k)}(\cdot, s,\theta)}(x) \dd{s}. 
\end{align*}
Note that this differentiation is valid by Fact \ref{fact:aaa} and properties of the Laplace transform (e.g. \cite[Chapter XIII.2]{Feller:1971cr}). We evaluate this expression at $\theta = 0$ 
\begin{align*}
	w^{(k)}(x,t,0) =& (-1)^k \T{t} { f^k(x) } \\
	&+ \lambda \intc{t} \T{t-s} \sbr{ p \sum_{l=1}^{k-1} \binom{k}{l} w^{(l)}(\cdot, s,0) w^{(k-l)}(\cdot, s,0) + (2p-1) w^{(k)}(\cdot, s,0)}(x) \dd{s}. 
\end{align*}
We recall that $\lambda_p=(2p-1)\lambda$. One easily checks that 
\begin{equation}
	w^{(k)}(x,t,0) = (-1)^k e^{\lambda_p t}\T{t} { f^k(x) } + \lambda \intc{t} e^{\lambda_p (t-s)} \T{t-s} \sbr{ p \sum_{l=1}^{k-1} \binom{k}{l} w^{(l)}(\cdot, s,0) w^{(k-l)}(\cdot, s,0)}(x) \dd{s}. \label{eq:moments-formula} 
\end{equation}
By the properties of the Laplace transform the moments are given by 
\begin{equation}
	\evx{} \rbr{\ddp{f}{X_t}}^n = (-1)^{n} w^{(n)}(x,t,0). \label{eq:moments} 
\end{equation}
By the formula above $w^{(n)}(x,t,0)$ can be made meaningful also in situations when the Laplace transform is not well-defined. In particular using standard techniques one can show that \eqref{eq:moments-formula} is valid for $f\in \pspace{}$ (in fact in this paper we will need only moments up to order $4$).

\subsection{Weak convergence facts} In this section we gather simple facts concerning the weak convergence. Let us denote by $\norm{\cdot}{TV}{}$ the total variation norm on the set of probability measures. We have a simple lemma 
\begin{lemma}
	\label{lem:totalVariation} Let $(\Omega, \mathcal{F}, \mathbb{P})$ be a probability space and $A_1\subset A_2 \in \mathcal{F}$ be such that $\pr{A_1} >0$. Let $X$ be a random variable and $\nu_1$, $\nu_2$ be its law conditioned on $A_1, A_2$ respectively. Then
	\[ \norm{\nu_1 - \nu_2}{TV}{} \leq 2 \frac{\pr{A_2} - \pr{A_1}}{ \pr{A_1}^2}. \]
\end{lemma}
\begin{proof}
	Let $B$ be a Borel set, then
	\begin{align*}
		& |\pr{X\in B | A_1} - \pr{X\in B| A_2}| = \Big|\frac{\pr{\cbr{X\in B} \cap A_1} }{\pr{A_1}} - \frac{\pr{\cbr{X\in B} \cap A_2} }{\pr{A_2}}\Big| \\
		= &\Big|\frac{ \pr{\cbr{X\in B} \cap A_1} (\pr{A_2} - \pr{A_1})+(\pr{\cbr{X\in B} \cap A_1} - \pr{\cbr{X\in B} \cap A_2}) \pr{A_1}}{\pr{A_1}\pr{A_2}}\Big| \\
		\leq& 2\frac{\pr{A_2} - \pr{A_1}}{ \pr{A_1}}. 
	\end{align*}
	
	The conclusion holds by the fact that $B$ is arbitrary. 
\end{proof}

Let $\mu_1, \mu_2$ be two probability measures on $\R$, and $\Lip(1)$ be the space of continuous functions $\R\mapsto [-1,1]$ with the Lipschitz constant smaller or equal to $1$. We define 
\begin{equation}
	m(\mu_1, \mu_2) := \sup_{g\in \Lip(1)} |\ddp{g}{\mu_1} - \ddp{g}{\mu_2}|. \label{eq:weakMetric} 
\end{equation}
It is well known that $m$ is a distance equivalent to weak convergence (see e.g. \cite[Theorem 11.3.3]{Dudley:2002}). One easily checks that when $\mu_1, \mu_2$ correspond to two random variables $X_1,X_2$ on the same probability space then we have 
\begin{equation}
	m(\mu_1, \mu_2) \leq \norm{X_1 - X_2}{1}{} \leq \sqrt{\norm{X_1 - X_2}{2}{}}. \label{eq:weakL2estimation} 
\end{equation}

\subsection{Rate of convergence to invariant measure and approximations} \label{sec:approximation}

We will need also estimations of the speed of convergence to the invariant measure. Throughout the proofs we will denote
\[ \tilde{f} := f - \ddp{f}{\eq}. \]
In proofs below it will be convenient to have some additional regularity conditions. Given $f\in\pspace{}$ and $u>0$ we denote 
\begin{equation}
	l_u(x) = \T{u} \tilde{f}(x). \label{eq:lu} 
\end{equation}
One easily checks that $l_u$ is a $C^\infty$ function and its derivatives grow at most polynomially. Moreover, we have $\ddp{l_u}{\eq}=0$. 

\begin{lemma}
	\label{lem:decay} Let $f\in \mathcal{P}(\Rd)$ and $\tilde{f}$ be defined as above. Then there exist constants $C,n>0$ such that for any $t\geq 1$ we have 
	\begin{equation}
		\T{t} \tilde{f}(x) \leq C (1+\norm{x}{}{n}) e^{-\mu t},\quad \T{t} \tilde{f}(0) \leq C e^{-2 \mu t}. \label{eq:decay1} 
	\end{equation}
	Moreover
	\[ \lim_{t\rightarrow +\infty} e^{\mu t}\T{t}\tilde{f}(x) = x\circ \ddp{\grad f }{\eq}, \]
	where $\grad f$ is understood in a weak sense.There exist $\tilde{C},\tilde{n}$ such that 
	\begin{equation}
		e^{\mu t} \T{t}\tilde{f}(x) - x\circ \ddp{\grad f }{\eq} \leq \tilde{C} (1+\norm{x}{}{\tilde{n}}) e^{-\mu t}.\label{eq:decay2} 
	\end{equation}
	Moreover, there exists a function $c:\R_+ \mapsto \R_+$ such that $c(u)\rightarrow 0$ as $u\searrow 0$ and 
	\begin{equation}
		\T{t} (l_u -\tilde{f}) (x) \leq c(u) (1+\norm{x}{}{n}) e^{-\mu t}. \label{eq:strongEstimation} 
	\end{equation}
\end{lemma}
\begin{proof}
	Let us recall equation \eqref{eq:semi-group}. We write 
	\begin{equation}
		e^{\mu (t+u)}\T{t+u} \tilde{f}(x)= e^{\mu u} e^{\mu t}\T{t} {l_u}(x) = e^{\mu u}\intr e^{\mu t}({l_u}(x e^{-\mu t} + ou(t) y ) -{l_u}(y)) \eq(y) \dd{y}.\label{eq:kuku1} 
	\end{equation}
	Using the mean value theorem we get 
	\begin{equation}
		h(x,y,t):=e^{\mu t}(l_u(x e^{-\mu t} + ou(t)y ) -l_u(y)) = (x + e^{\mu t} (ou(t)-1)y) \circ \grad l_u(x_0),\label{eq:kuku2} 
	\end{equation}
	where $x_0$ is some point on the interval joining $y$ and $x e^{-\mu t} + ou(t)y$. From this representation we obtain $|h(x,y,t)|\cleq \max(\norm{x}{}{n},\norm{y}{}{n})$ and $|h(0,y,t)|\leq e^{-\mu t} \norm{y}{}{n}$. This is enough to show (\ref{eq:decay1}).
	
	We notice that $h(x,y,t) \rightarrow \sum_{i=1}^d x_i \frac{
	\partial l_u}{
	\partial y_i}$ point-wise as $t\rightarrow +\infty$. This, together with the Lebesgue dominated convergence yields
	\[ \lim_{t\rightarrow +\infty} e^{\mu t}\T{t}\tilde{f}(x) = e^{\mu u} \sum_{i=1}^d x_i \ddp{\frac{
	\partial l_u}{
	\partial y_i}}{\eq} = e^{\mu u} \sum_{i=1}^d x_i \ddp{l_u}{\frac{
	\partial \eq}{
	\partial y_i}} = (*) \]
	The calculations above are valid for any choice of $u$, in particular, letting $u\searrow 0$ we obtain
	\[ (*) = \sum_{i=1}^d x_i \ddp{f}{\frac{
	\partial \eq}{
	\partial y_i}} = \sum_{i=1}^d x_i \ddp{\frac{
	\partial f}{
	\partial y_i}}{\eq} = x\circ \ddp{\grad f }{\eq}. \]
	For the sake of notational simplicity \eqref{eq:decay2} will be proved for $d=1$. By the above calculations
	\[ e^{\mu (t+u)} \T{t+u}\tilde{f}(x) - x \ddp{ f' }{\eq} = e^{\mu u} \rbr{ e^{\mu t} \T{t}l_u(x) - x \ddp{ l_u' }{\eq}}. \]
	Let us now treat the inner expression using \eqref{eq:semi-group}: 
	\begin{align*}
		e^{\mu t} \T{t}l_u(x) - x \ddp{ l_u' }{\eq} &= \ev{} \rbr{e^{\mu t}\rbr{l_u(x e^{-\mu t} + ou(t)G) - l_u(G)} - x l_u'(G) }\\
		&= \ev{} \rbr{e^{\mu t} \int^{x e^{-\mu t} + ou(t)G}_G l_u'(y) \dd{y} - x l_u'(G) } \\
		&= \ev{} \rbr{e^{\mu t} \int^{x e^{-\mu t} + ou(t)G}_G (l_u'(y) - l_u'(G)) \dd{y}} + {e^{\mu t} (ou(t)-1) \ev{} G l_u'(G)}. 
	\end{align*}
	The second term is easily upper-bounded by $C e^{-\mu t}$, for some $C>0$. The first one can be rewritten as
	\[ \ev{} e^{\mu t} \int^{x e^{-\mu t} + ou(t)G}_G \int_G^{y} l_u''(z) \dd{z} \dd{y} \leq \ev{} e^{\mu t} (x e^{-\mu t} + (ou(t)-1)G)^2 \sup_{z \in [G, xe^{-\mu t} + ou(t) G]} |l_u''(z)| = (*). \]
	By the discussion at the beginning of the proof we know that $l_u''$ grows polynomially, i.e. there exists $n$ such that
	\[ (*) \cleq e^{-\mu t}\ev{} (x + e^{\mu t} (ou(t)-1)G)^2 \max(\norm{x}{}{} + \norm{G}{}{} )^n \cleq e^{-\mu t} (1+\norm{x}{}{n}). \]
	This is enough to prove \eqref{eq:decay2}.
	
	Now, we need an estimation of $e^{\mu t}\T{t} (l_u - f)(x)$ which takes into account $u$. Using the same trick as before it is enough to prove an estimation for $e^{\mu t}\T{t} (l_{u+1} - l_1)(x)$. Denoting $k_u(x) = l_{u+1}(x) - l_1(x)$ we have
	\[ e^{\mu t}\T{t}k_u(x) = \intr e^{\mu t}({k_u}(x e^{-\mu t} + ou(t) y ) -{k_u}(y)) \eq(y) \dd{y}. \]
	Using the mean value theorem we get $$h(x,y,t):=e^{\mu t}(l_u(x e^{-\mu t} + ou(t)y ) -l_u(y)) = (x + e^{\mu t} (ou(t)-1)y) \circ \grad l_u(x_0).$$ Further following arguments used in \eqref{eq:kuku1} and \eqref{eq:kuku2} and after them we conclude that in order to obtain \eqref{eq:strongEstimation} we need to estimate $\grad (l_{u+1} - l_1)$. For simplicity we will provide the details for $d=1$. We have
	\[ \rbr{(l_{u+1} - l_1)(x)}' = \rbr{\int_{\R} (l_1(x-y)- l_1(x) ) g_u(y) \dd{y} }' = {\int_{\R} (l_1'(x-y) -l'(x)) g_u(y) \dd{y} }=(*). \]
	Using the mean value theorem there exists a function $x_0(x,y)$ such that $x_0(x,y)\in [x-y,x]$ and
	\[ (*) = \int_{\R} y l''_1(x_0(x,y)) g_u(y) \dd{y} =(**). \]
	The function $l''_1$ is grows at most polynomially. Therefore for some $c,n$ we have
	\[ |(**)| \leq c \intr |y| (1+\norm{x}{}{n} + \norm{y}{}{n}) g_u(y) \dd{y} \leq c(u) (1+\norm{x}{}{n}), \]
	where $c(u)$ is some function fulfilling the required conditions. 
\end{proof}

Let us denote $Y_t(f) := F_t \rbr{\ddp{X_t}{\tilde{f} } - |X_t| \ddp{\tilde{f}}{\eq} }$, where $F_t= e^{-(\lambda_p/2)t }$ when $\lambda_p<2\mu$; $F_t= e^{-(\lambda_p/2)t }t^{-1/2}$ when $\lambda_p=2\mu$; $F_t= e^{-(\lambda_p-\mu )t }$ when $\lambda_p>2\mu$. Let us assume that we proved 
\begin{equation}
	\limsup_{t\rightarrow +\infty} \ev{} \rbr{Y_t(f) - Y_t(l_u)}^2 = c(u), \label{eq:assumptionUniform} 
\end{equation}
where $c(u)$ is some function such that $c(u) \searrow 0$ as $u\searrow 0$. Let $L(l_u), L(f)$ be the laws of the limits asserted in Theorem \ref{thm:clt1} and Theorem \ref{thm:cltCritical}.

Assume that we know already that convergences in Theorem \ref{thm:clt1} and Theorem \ref{thm:cltCritical} hold for any $l_u,u>0$. One can check that in either case
\[ m(L(l_u), L(f)) \rightarrow 0,\quad \text{ as }u \searrow 0. \]
Let us fix $\epsilon>0$ and choose $u>0$ such that $c(u)\leq \epsilon^2$ and $m(L(l_u), L(f))\leq \epsilon$. We can choose $T_1>0$ such that for any $t>T_1$ we have $m(L(l_u), \mathcal{L}(Y_t(l_u))) \leq \epsilon$ (where $\mathcal{L}$ denotes the law of given random variable). Let $T_2$ be such that for any $t>T_2$ the $L_2$ norm in \eqref{eq:assumptionUniform} is smaller than $4\epsilon^2$. Now, for $t>\max(T_1,T_2)$ we have $m(L(f), Y_t(f)) \leq 4\epsilon$ and therefore all of theorems mentioned hold also for $f$. The case of Theorem \ref{thm:clt2} follows similarly, but one has to use convergences in probability and the $L^2$-norm.

Let us sum up this section. In the proofs below we may work with additional assumption that $f$ is a smooth function and its derivatives grow at most polynomially fast if only we show \eqref{eq:assumptionUniform}.

\subsection{LLN and CLT for small branching rate} 

\label{sub:proof_of_theorem_thm:clt1}

First we prove the following fact 
\begin{fact}
	\label{fact:momentsSubcritical} Let $\cbr{X_t}_{t\geq 0}$ be the OU branching system and $\lambda_p < 2 \mu$ and let $f\in \pspace{}$. Then
	
	\[ \evx{} \rbr{e^{-(\lambda_p/2) t} \ddp{X_t}{\tilde{f}}} \rightarrow 0 \quad \text{ as } t\rightarrow +\infty. \]
	\begin{equation}
		\evx{} \rbr{e^{-(\lambda_p/2) t} \ddp{X_t}{\tilde{f}}}^2 \rightarrow \sigma_f^2, \quad \var_x\rbr{e^{-(\lambda_p/2) t} \ddp{X_t}{\tilde{f}}} \rightarrow \sigma_f^2, \label{eq:subcriticalSecondMoments} 
	\end{equation}
	where $\sigma_f^2$ is the same as in \eqref{eq:sigmaf}. Moreover,
	\[ \sup_t \evx{} \rbr{e^{-(\lambda_p/2)t}\rbr{\ddp{X_t}{\tilde{f}} - \evx{}\ddp{X_t}{\tilde{f}}} }^4 <+\infty. \]
\end{fact}
\begin{proof}
	First we note that by (\ref{eq:moments-formula}) and Lemma \ref{lem:decay} (ineq. (\ref{eq:decay1})),
	\[ |w'(x,t,0)| \cleq e^{(\lambda_p -\mu)t} (1+ \norm{x}{}{n}) \leq e^{(\lambda_p/2)t} (1+ \norm{x}{}{n}), \]
	which, by \eqref{eq:moments}, implies the first assertion. Using (\ref{eq:moments-formula}) and \eqref{eq:moments} again we calculate the second moment 
	\begin{multline*}
		\evx{} \rbr{e^{-(\lambda_p/2)t}\ddp{X_t}{\tilde{f}}}^2 = \T{t} \tilde{f}^2(x) + 2\lambda p e^{-\lambda_p t} \intc{t} e^{\lambda_p (t-s)} \T{t-s} \sbr{\rbr{ e^{\lambda_p s}\T{s} \tilde{f}(\cdot)}^2}(x) \dd{s} \\
		= \T{t} \tilde{f}^2(x) + 2\lambda p \intc{t} \T{t-s} \sbr{\rbr{ e^{(\lambda_p/2) s}\T{s} \tilde{f}(\cdot)}^2}(x) \dd{s}. 
	\end{multline*}
	By \eqref{eq:decay1} in Lemma \ref{lem:decay} the integrand in the last expression can be estimated as follows
	\[ \T{t-s} \sbr{\rbr{ e^{(\lambda_p/2) s}\T{s} \tilde{f}(\cdot)}^2}(x) \cleq e^{(\lambda_p - 2\mu) s} \T{t-s} \sbr{(1+ \norm{\cdot}{}{n} )^2}(x). \]
	Using representation \eqref{eq:OUrep} it can be checked that for any $t\geq 0$ we have $\T{t} \sbr{(1+ \norm{\cdot}{}{n})^2}(x) \cleq (1+ \norm{x}{}{2n})$. The dominated Lebesgue theorem implies \eqref{eq:subcriticalSecondMoments}. We also conclude that for any $t\geq 0$,
	\[ w''(x,t,0) \cleq e^{\lambda_p t}(1+ \norm{x}{}{2n}). \]
	Similarly we investigate $w'''(x,t,0)$. By \eqref{eq:moments-formula} we have
	\[ |w'''(x,t,0)| \cleq e^{\lambda_p t}\T{t} { |\tilde{f}|^3(x) } + \left.\intc{t} e^{\lambda_p (t-s)} \T{t-s} \sbr{ w''(\cdot, s,0) w'(\cdot, s,0)}(x) \dd{s} .\right. \]
	Using the estimates obtained already and the fact that $\T{t}\sbr{(1+\norm{\cdot}{}{n})^3}(x)\cleq (1+\norm{x}{}{3n})$
	\[ |w'''(x,t,0)| \cleq (1+\norm{x}{}{3n})e^{\lambda_p t} + e^{\lambda_p t} \left.\intc{t} e^{\lambda_p/2 s} \T{t-s} \sbr{ 1+ \norm{\cdot}{}{3n} }(x) \dd{s} \right. \cleq e^{(3/2) \lambda_p t} (1+ \norm{x}{}{3n}). \]
	Finally, we will also need the fourth moment. By \eqref{eq:moments-formula} and the estimates above we get 
	\begin{align*}
		\evx{} &\rbr{e^{-(\lambda_p/2)t}\ddp{X_t}{\tilde{f}}}^4 \\
		&\cleq e^{-\lambda_p t}\T{t} { \tilde{f}^4(x) } + e^{-2\lambda_p t} \intc{t} e^{\lambda_p (t-s)} \T{t-s} \sbr{ w''(\cdot,s,0)^2 + w'''(\cdot,s,0) w'(\cdot,s,0) }(x) \dd{s} \\
		&\cleq e^{-\lambda_p t} (1+\norm{x}{}{4n})+ e^{-\lambda_p t} \intc{t} e^{\lambda_p s} \T{t-s} \sbr{ (1 + \norm{\cdot}{}{4n}) }(x) \dd{s} \cleq (1 + \norm{x}{}{4n}). 
	\end{align*}
	It is now easy to get the last assertion of the fact. 
\end{proof}
Now we will prove representation \eqref{eq:sigmafSubcriticalHermite}. We have $\T{s} H_{i_1, i_2, \ldots, i_d}(x) = e^{-(i_1+i_2 +\ldots +i_d)\mu t} H_{i_1,i_2,\ldots,i_d} (x)$. Therefore 
\begin{multline*}
	\sigma_f^2 := \ddp{\eq}{\rbr{ \sum_{i_1=0, i_2 =0,\ldots, i_d=0}^{+\infty} f_{i_1, i_2, \ldots, i_d} H_{i_1, i_2, \ldots, i_d} }^2} \\
	+ 2\lambda p \inti \ddp{\eq}{\rbr{ e^{(\lambda_p/2) s} \sum_{i_1=0, i_2 =0,\ldots, i_d=0}^{+\infty} e^{-\mu(i_1 +i_2+\ldots +i_d)s}f_{i_1, i_2, \ldots, i_d} H_{i_1, i_2, \ldots, i_d}(\cdot)}^2} \dd{s} \\
	= \sum_{i_1=0, i_2 =0,\ldots, i_d=0}^{+\infty} f_{i_1, i_2, \ldots, i_d}^2 + 2\lambda p \sum_{i_1=0, i_2 =0,\ldots, i_d=0}^{+\infty} f_{i_1, i_2, \ldots, i_d}^2 \inti e^{(\lambda_p-2(i_1+i_2+\ldots+i_d) \mu) s} \dd{s}. 
\end{multline*}
Now \eqref{eq:sigmafSubcriticalHermite} follows easily. We are ready for 
\begin{proof}
	[Proof of Theorem \ref{thm:LLN} (sketch)] The proof of \cite[Theorem 6]{Englander:2007im} can be checked to hold also for the branching mechanism introduced in our paper. We will show now that the almost sure convergence holds also for bounded functions which are not compactly supported. We decompose $f = f^+ - f^-$, where $f^+(x) = f(x) 1_{ \cbr{f(x) \geq 0}}$ and $f^+(x) = -f(x) 1_{ \cbr{f(x) < 0}}$. It is enough to prove the claim separately for $f^+$ and $f^-$, hence we assume that $f \geq 0$. For $n \in \mathbb{N}$ we consider functions $h_n(x) := \min(\max(n - x, 0),1) $ and $g_n:\R^d \mapsto [0,1]$ given by $g_n(x):= h_n( \norm{x}{}{})$. By \cite[Theorem 6]{Englander:2007im} we know that $e^{-\lambda_p t} \ddp{X_t}{g_n} \rightarrow V_\infty \ddp{g_n}{\eq} \: a.s.$, we know also that $e^{-\lambda_p t} |X_t| \rightarrow V_\infty \: a.s. $ Therefore $e^{-\lambda_p t} \ddp{X_t}{(1-g_n)} \rightarrow V_\infty \ddp{\eq}{(1-g_n)}\: a.s$. Now we estimate
	\[ \ddp{X_t}{ f g_n} \leq \ddp{X_t}{ f} \leq \ddp{X_t}{ f g_n} + \norm{f}{\infty}{} \ddp{X_t}{(1-g_n)}. \]
	Using the previous considerations and the fact that $f g_n$ has compact support we get with probability one
	\[ \ddp{ f g_n}{\eq} V_\infty \leq \liminf_{t\nearrow +\infty} e^{-\lambda_p t} \ddp{X_t}{ f} \leq \limsup_{t\nearrow +\infty} e^{-\lambda_p t} \ddp{X_t}{ f} \leq \ddp{ f g_n}{\eq} V_\infty + \norm{f}{\infty}{} \ddp{ (1-g_n)}{\eq} V_\infty. \]
	To conclude, we observe that $\ddp{ f g_n}{\eq} \rightarrow \ddp{ f}{\eq} $ and $\ddp{ f g_n}{\eq} + \norm{f}{\infty}{} \ddp{ (1-g_n)}{\eq} \rightarrow \ddp{f}{\eq}$ as $n\rightarrow +\infty$.\\
	Let now $f\in \pspace{}$ and $\tilde{f} := f - \ddp{f}{\eq}$. Calculating similarly as in the proof of Fact \ref{fact:momentsCritical} we get 
	\begin{equation}
		e_t:=\evx{} \rbr{e^{-\lambda_p t}\ddp{X_t}{\tilde{f}}}^2 \cleq e^{-\lambda_p t} \T{t} \tilde{f}^2(x) + \intc{t} e^{-\lambda_p s} \T{s} \sbr{\rbr{ \T{t-s} \tilde{f}(\cdot)}^2}(x) \dd{s}. \label{eq:LLNL2Convergence} 
	\end{equation}
	Obviously for any $x\in \R^d $ we have $\T{t} \tilde{f}(x) \rightarrow 0$, moreover $|\T{t} \tilde{f}(x)|\cleq 1+\|x\|^n$. Standard considerations using the Lebesgue dominated convergence theorem yield that $e_t \rightarrow 0$. Therefore $e^{-\lambda_p t} \ddp{X_t}{\tilde{f}} \rightarrow^P 0$, and further $e^{-\lambda_p t} \ddp{X_t}{f} - e^{-\lambda_p t}|X_t|\ddp{f}{\eq} \rightarrow^P 0 $, which concludes the proof. 
\end{proof}

In the proofs below we will use the term ``X is asymptotically equivalent Y'' to denote the situation that $X_t - Y_t \rightarrow 0$ in probability (equivalently in law) as $t\rightarrow +\infty$. Now we are ready for 
\paragraph{Proof of Theorem \ref{thm:clt1}.} We start with the following random vector
\[ Z_1(t):=\rbr{e^{-\lambda_p t} |X_t|, e^{-(\lambda_p/2)t} (|X_t| - e^{\lambda_p t} V_\infty), e^{-(\lambda_p/2)t} \ddp{X_t}{\tilde{f}} }. \]
Let $n\in \mathbb{N}$ to be fixed later and let us write
\[ Z_1(n t):=\rbr{e^{- n \lambda_p t} \ddp{X_{nt}}{1}, e^{-(n\lambda_p/2)t} (|X_{nt}| - e^{n\lambda_p t} V_\infty) , e^{-(n\lambda_p/2)t} \sum_{i=1}^{|X_t|} \ddp{X^{i,t}_{(n-1)t}}{\tilde{f}} }, \]
where $\cbr{X^{i,s}_t}_t$ denotes the subsystem originating from the particle $X_s(i)$. We know that with probability one $e^{-\lambda_p t} |X_t| \rightarrow V_\infty$. Therefore $e^{-n\lambda_p t} |X_{nt}| -e^{-\lambda_p t} |X_{t}| \rightarrow 0\:a.s.$ as $t\rightarrow +\infty$. Let us consider the second term 
\begin{multline*}
	|X_{nt}| - e^{n\lambda_p t} V_\infty = |X_{nt}| - e^{n\lambda_p t} \lim_{s\rightarrow +\infty } e^{-\lambda_p (nt+s)}|X_{nt+s}| = |X_{nt}| - e^{n\lambda_p t} \lim_{s\rightarrow +\infty } e^{-\lambda_p (s + nt)} \sum_{i=1}^{|X_{nt}|} |X^{i,nt}_s|\\
	= |X_{nt}| - \sum_{i=1}^{|X_{nt}|} \lim_{s\rightarrow +\infty } e^{-\lambda_p s } |X^{i,nt}_s| = \sum_{i=1}^{|X_{nt}|} \rbr{1 - V_\infty^i}, 
\end{multline*}
where $V_\infty^i$ are independent copies of $V_\infty$ (note that formally they depend on $t$, however we will suppress this fact in the notation). We couple each $X^{i,t}$ with the branching system starting from one particle located at $0$. To this end we use the same methods as in Fact \ref{fact:coupling} for particles movements and retain the branching structure. The coupled system is denoted by $\tilde{X}^{i,t}$. Let us write 
\begin{multline*}
	H:=\left| e^{-(n\lambda_p/2)t} \sum_{i=1}^{|X_t|} \ddp{X^{i,t}_{(n-1)t}}{\tilde{f}} - e^{-(n\lambda_p/2)t} \sum_{i=1}^{|X_t|} \ddp{\tilde{X}^{i,t}_{(n-1)t}}{\tilde{f}}\right| \\
	\leq \left|e^{-(n\lambda_p/2)t} \sum_{i=1}^{|X_t|} 1_{\norm{X_t(i)}{}{}<t} \sum_{j=1}^{|X^{i,t}_{(n-1)t}|}\Big( \tilde{f}(X^{i,t}_{(n-1)t}(j)) -\tilde{f}(\tilde{X}_{(n-1)t}^{i,t}(j))\Big) \right| \\+ e^{-(n\lambda_p/2)t} \sum_{i=1}^{|X_t|} 1_{\norm{X_t(i)}{}{}\geq t} \left|\ddp{X^{i,t}_{(n-1)t}}{\tilde{f}} +\ddp{\tilde{X}^{i,t}_{(n-1)t}}{\tilde{f}} \right|. 
\end{multline*}
By the approximation argument from Section \ref{sec:approximation} we may assume that $\tilde{f}$ and its derivative can be upper-bounded by a polynomial of order $k \in \mathbb{N}$. Indeed \eqref{eq:assumptionUniform} holds by \eqref{eq:subcriticalSecondMoments} and easy calculations. 

By Fact \ref{fact:coupling} and the mean value theorem we have $|\tilde{f}(X^{i,t}_{(n-1)t}(j)) -\tilde{f}(\tilde{X}_{(n-1)t}^{i,t}(j))| 1_{\norm{X_t(i)}{}{}\leq t} \cleq t e^{-(n-1)\mu t} (1+\norm{\tilde{X}_{(n-1)t}^{i,t}(j)}{}{})^k$. Using the conditional expectation with respect to $X_t$, \eqref{eq:moments} and $\T{t} (1+\norm{\cdot}{}{k})(x)\cleq (1+\norm{x}{}{k})$ we get 
\begin{align*}
	\evx{} H \cleq& e^{-(n\lambda_p/2)t}\evx{} \sum_{i=1}^{|X_t|}(1+\norm{X_t(i)}{}{})^k t e^{-(n-1)\mu t} e^{(n-1)\lambda_p t} \\
	&+ e^{-(n\lambda_p/2)t}\evx{} \sum_{i=1}^{|X_t|} 1_{\norm{X_t(i)}{}{}\geq t}(1+\norm{X_t(i)}{}{})^k e^{(n-1)\lambda_p t} \\
	\cleq& e^{(n\lambda_p/2)t} t e^{-(n-1)\mu t} + e^{((n-2)\lambda_p/2)t} \T{t} ((1+\norm{\cdot}{}{})^k 1_{ \norm{\cdot}{}{} >t})(x). 
\end{align*}
There exists $n_0>0$ such that for any $n>n_0$ we have $n\lambda_p < 2\mu (n-1)$. Using the Schwarz inequality, the second term can be estimated by $\sqrt{\T{t} (1+\norm{\cdot}{}{k})^2 (x)}\sqrt{\T{t} 1_{ \norm{\cdot}{}{}>t} (x)}$. The Ornstein-Uhlenbeck process has Gaussian marginals with bounded mean and variance therefore $\T{t} 1_{ \norm{\cdot}{}{}>t} (x) \cleq e^{-c t^2}$ for a certain $c>0$. We may conclude that $\evx{} H \rightarrow 0$ as $t\rightarrow +\infty$. We notice that $\tilde{X}^{i,t}$'s are i.i.d. branching particle systems; indeed the only dependence among ${X}^{i,t}$ is by the initial condition. Let us denote $Z_t^i := e^{-((n-1)\lambda_p/2) t } \ddp{\tilde{X}^{i,t}_{(n-1)t}}{\tilde{f}}$ and $z_t^i := \ev{}_0 Z_t^i$ (we will write $\ev{}_0$ to underline the fact that $\tilde{X}^{i,t}$ starts from $0$).

By \eqref{eq:moments} and Lemma \ref{lem:decay} one checks easily that for $n$ large enough
\[ \evx{} e^{-(\lambda_p/2)t} \sum_{i=1}^{|X_t|} |z^i_t| = e^{(\lambda_p/2)t} |\ev{}_0 Z_t^1| \cleq e^{(n\lambda_p/2)t} e^{-(n-1)\mu t} \rightarrow 0, \quad \text{as } t\rightarrow +\infty. \]
Using the facts above we conclude that there exists $n$ such that $Z_1(nt)$ is asymptotically equivalent to
\[ Z_2(t) = \rbr{e^{-\lambda_p t} |X_t|, e^{-(n\lambda_p/2)t} \sum_{i=1}^{|X_{nt}|} \rbr{1- V_\infty^i}, e^{-(\lambda_p/2)t} \sum_{i=1}^{|X_t|} (Z^i_t - z^i_t) }. \]
Let us now denote
\[ Z_3(t):=\rbr{e^{-\lambda_p t} |X_t|, {|X_{nt}|}^{-1/2} \sum_{i=1}^{|X_{nt}|} \rbr{1- V_\infty^i}, {|X_{t}|}^{-1/2} \sum_{i=1}^{|X_t|} (Z^i_t - z^i_t) }, \]
which we will consider conditionally on the event $\cbr{|X_t| \neq 0}$ (here and below we adopt the convention that $a/0 = 0$, $Z_3$ is well defined then). The corresponding expected value is denoted by $\evx{}_{,|X_t|\neq 0}$. Let us denote the characteristic function of $Z_3$ 
\begin{align*}
	& \chi_1(\theta_1, \theta_2,\theta_3;t) := \\
	&\evx{}_{,|X_t|\neq 0} \exp\cbr{i \theta_1 e^{-\lambda_p t} |X_t| + i \theta_2 {|X_{nt}|}^{-1/2} \sum_{i=1}^{|X_{nt}|} \rbr{1- V_\infty^i}+ i\theta_3|X_t|^{-1/2} \sum_{i=1}^{|X_t|} (Z^i_t - z^i_t) }. 
\end{align*}
Conditioning on $X_{nt}$ and using the Markov property we check that variables $1-V_\infty^i$ are i.i.d, moreover they are independent of the system before time $nt$. We denote their common characteristic function by $h$.

We have
\[ \chi_1(\theta_1, \theta_2,\theta_3;t) = \evx{}_{,|X_t|\neq 0} \exp\cbr{i \theta_1 e^{-\lambda_p t} |X_t| + i\theta_3|X_t|^{-1/2} \sum_{i=1}^{|X_t|} (Z^i_t - z^i_t) } h \rbr{\theta_2/\sqrt{|X_{nt}|}} ^{|X_{nt}|}. \]
By Fact \ref{fact:aaa} and the central limit theorem we write $h(\theta_2/\sqrt{n})^n \rightarrow e^{-\theta_2^2/(2 \sigma_V^2)}$, where $\sigma_V^2=\frac{1}{2p-1}$. Further we will work with 
\begin{align*}
	& \chi_2(\theta_1, \theta_2,\theta_3;t) := \\
	&\evx{}_{,|X_t|\neq 0} \exp\cbr{i \theta_1 e^{-\lambda_p t} |X_t| + i\theta_3|X_t|^{-1/2} \sum_{i=1}^{|X_t|} (Z^i_t - z^i_t) } h \rbr{\theta_2/\sqrt{e^{(n-1)\lambda_p t}|X_{t}|}}^{e^{(n-1)\lambda_p t}|X_{t}|}. 
\end{align*}
The limit of $\chi_1$ is the same as the one of $\chi_2$ providing that any of them exists. Indeed 
\begin{align*}
	&|\chi_2(\theta_1, \theta_2,\theta_3;t) - \chi_1(\theta_1, \theta_2,\theta_3;t)| \\
	&\leq \evx{}_{,|X_t|\neq 0} \left|h\rbr{\theta_2/\sqrt{e^{(n-1)\lambda_p t}|X_{t}|}}^{e^{(n-1)\lambda_p t}|X_{t}|} - h \rbr{\theta_2/\sqrt{|X_{nt}|}} ^{|X_{nt}|}\right| \\
	&= \evx{}_{,|X_t|\neq 0} \left| \ldots \right| 1_{Ext} + \evx{}_{,|X_t|\neq 0} \left| \ldots \right| 1_{Ext^c}. 
\end{align*}
The sequence of events $\cbr{|X_t|\neq 0}$ decreases to $Ext^c$ and $\pr{Ext^c}>0$. We have $|h| \leq 1$ hence the first summand converges to $0$. By Fact \ref{fact:aaa} on $Ext^c$ we have $\frac{|X_{nt}|}{|X_t| e^{(n-1)\lambda_p t}} \rightarrow 1\:a.s.$ therefore the second summand converges to $0$ as well. Now we recall that $Z^i_t$ are i.i.d and independent of $X_t$. Conditioning with respect to $|X_t|$ we get 
\begin{align}
	& \chi_2(\theta_1, \theta_2,\theta_3;t) = \nonumber\\
	&\evx{}_{,|X_t|\neq 0} \exp\cbr{i \theta_1 e^{-\lambda_p t} |X_t|} h_{Z_{t}^1 - z_{t}^1}\rbr{\theta_3/ |X_t|^{1/2}}^{|X_t|} h \rbr{\theta_2/\sqrt{e^{(n-1)\lambda_p t}|X_{t}|}}^{e^{(n-1)\lambda_p t}|X_{t}|}, \label{eq:tmp99} 
\end{align}
where $h_{Z_{t}^1 - z_{t}^1}$ denotes the characteristic function of $Z_t^1 - z_t^1$. We will now investigate the second term of the product above. To this end we fix two sequences $\cbr{t_m}$ and $\cbr{a_m}$ such that $t_m\rightarrow +\infty, a_m\rightarrow +\infty,\text{ as } m\rightarrow +\infty$ and consider
\[ S_m := \frac{1}{s_m} \sum_{i=1}^{a_m} (Z_{t_m}^i - z_{t_m}), \]
where $Z_{t_n}^i$ are i.i.d. copies as described above and $s_m^2 := a_m \var_0(Z_{t_m}^1) $. From Fact \ref{fact:momentsSubcritical} we easily see that $s_m^2$ is of the same order as $a_m$. We will now verify the Lindeberg condition. Let us calculate
\[ L_m(r) := \frac{1}{s_m^2} \sum_{i=1}^{a_m} \ev{}_0\rbr{(Z_{t_m}^i - z_{t_m}^i )^2 1_{|Z_{t_m}^i - z_{t_m}^i| > r s_m} } = \frac{a_m}{s_m^2} \ev{}_0\rbr{(Z_{t_m}^1 - z_{t_m}^1 )^2 1_{|Z_{t_m}^i - z_{t_m}^i| > r s_m} }. \]
By by the Schwarz inequality, the Chebyshev inequality and Fact \ref{fact:momentsSubcritical} we get
\[ L_m(r) \leq \frac{a_m}{s_m^2} \sqrt{\pr{|Z_{t_m}^i - z_{t_m}^i| > r s_m}} \sqrt{ \ev{}_0 (Z_{t_m}^1 - z_{t_m}^1 )^4} \cleq \frac{a_m}{s_m^2} \frac{{\ev{}_0{(Z_{t_m}^i - z_{t_m}^i)^4}}}{(r s_m)^2} \cleq \frac{a_m}{s_m^2} \frac{1}{(r s_m)^2} \rightarrow 0, \]
as $m\rightarrow +\infty$. The CLT holds therefore $S_m\rightarrow^d \mathcal{N}(0,1)$. Obviously $\frac{s_m^2}{a_m} \rightarrow \sigma_f^2$. We conclude that 
\begin{equation}
	h_{Z_{t_m} - z_{t_m}}(\theta_3 / \sqrt{a_m})^{a_m} \rightarrow e^{-\theta_3^2/(2 \sigma_f^2)}. \label{eq:tmp82} 
\end{equation}
Now we go back the $\chi_2$
\[ \chi_2(\theta_1, \theta_2, \theta_3;t) = \evx{}_{,|X_t|\neq 0} (\ldots) 1_{Ext} + \evx{}_{,|X_t|\neq 0} (\ldots) 1_{Ext^c}. \]
Arguing as before we get that the first summand converges to $0$. The second one is
\[ \evx{}_{,|X_t|\neq 0} (\ldots) 1_{Ext^c} = \mathbb{P}_{x,|X_t| \neq 0} (Ext^c) \evx{}_{,Ext^c} (\ldots), \]
where $\evx{}_{,Ext^c}$ denotes the conditional probability with respect to $Ext^c$. We have $\mathbb{P}_{|X_t| \neq 0} (Ext^c)\rightarrow 1$. We proved \eqref{eq:tmp82} for arbitrary sequences therefore the second term in \eqref{eq:tmp99} converges (on the set of non-extinction) as follows
\[ h_{Z_{t}^1 - z_{t}^1}\rbr{\theta_3/ |X_{t}|^{1/2}}^{|X_{t}|} \rightarrow e^{-\theta_3^2/(2 \sigma_f^2)},\quad a.s.,\:\: \text{ as } t\rightarrow +\infty. \]
We know that $e^{-\lambda_p t} |X_{t}| \rightarrow W$ a.s. (recall that $W$ denotes $V_\infty$ conditioned on $Ext^c$). By the Lebesgue dominated convergence theorem we get
\[ \chi_2(\theta_1, \theta_2,\theta_3;t) \rightarrow e^{-\theta_2^2/(2 \sigma_V^2)} e^{-\theta_3^2/(2 \sigma_f^2)} \evx{} \exp\cbr{i \theta_1 W}, \text{ as }t\rightarrow +\infty. \]
We know that $\cbr{|X_t| \neq 0}$ is a decreasing sequence of events and $\bigcap_{i=1}^\infty \cbr{|X_t| \neq 0} = Ext^c$. Moreover $\pr{Ext^c}>0$, therefore by Lemma \ref{lem:totalVariation} the limit of $Z_3(t)$ is the same if considered conditioned with respect to $\cbr{|X_t| \neq 0}$ or conditioned with respect to $Ext^c$. In this way we have proved that conditionally on $Ext^c$
\[ Z_3(t) \rightarrow^d \rbr{W, G_1, G_2}, \]
where $\rbr{W, G_1, G_2}$ is the same as in the statement of Theorem \ref{thm:clt1}. Using some obvious transformations we get
\[ Z_2(t) \rightarrow^d \rbr{W, \sqrt{W} G_1, \sqrt{W} G_2} \:\text{ and }\: Z_1(t) \rightarrow^d \rbr{W, \sqrt{W} G_1, \sqrt{W} G_2}, \]
which easily implies the convergence asserted in the theorem.

\subsection{CLT for large branching rate} We start with 
\begin{proof}
	[Proof of Fact \ref{fact:martingaleConvergence}] The fact that $H$ is a martingale with respect to the filtration of the OU branching system follows by its branching property and easy calculations using (\ref{eq:moments-formula}), \eqref{eq:moments} and \eqref{eq:semi-group}. Consider now $f(x) = \|x\|$. Using \eqref{eq:moments-formula} and \eqref{eq:moments} we can get 
	\begin{align*}
		\evx{} H_t^2 &\ceq e^{2(-\lambda_p + \mu) t} e^{\lambda_p t} \T{t} f^2(x) + e^{2(-\lambda_p + \mu) t}\intc{t} e^{\lambda_p (t-s)} \T{t-s}\sbr{ e^{2\lambda_p s} (T_s f(\cdot))^2 } (x) \dd{s}\\
		&\cleq e^{(-\lambda_p + 2\mu)t}(1 +\|x\|^2) + e^{(-\lambda_p + 2\mu)t}\intc{t} e^{\lambda_p s} \T{t-s}\sbr{ e^{-2\mu s} f^2 } (x) \dd{s} \\
		&\cleq (1+\|x\|^2) \rbr{1 + \intc{t} e^{(-\lambda_p + 2\mu)(t-s)} \dd{s}} \leq C (1+\|x\|^2). 
	\end{align*}
	To check that $V$ and $H$ are orthogonal one can use \eqref{eq:moments} together with the polarisation formula. We leave this task to the reader, just mentioning that this result is somehow expected as $1,x$ are orthogonal with respect to densities $g_t$ defining \eqref{eq:OUrep}. 
\end{proof}
\begin{proof}
	[Proof of Fact \ref{fact:law}] The proof follows easily by Fact \ref{fact:coupling}. Indeed it is enough to notice that, if we couple $X$ and $\tilde{X}$ according to the procedures described above, we may write
	\[ \tilde{H}_t - H_t = e^{(-\lambda+\mu)t} x e^{-\mu t} |X_t| = e^{-\lambda t} x |X_t|, \]
	where $\tilde{H},H$ are defined according to \eqref{eq:martingale}. 
\end{proof}

It will be useful to know that the range of the OU branching system grows at most linearly. This is a well-known fact fact for the branching Brownian motion. It is somehow obvious that it also holds for the OU branching motion as the Ornstein-Uhlenbeck process is ``better concentrated'' than the Brownian motion. However, we were not able to find a proof in the literature hence we provide one.
\begin{lemma}
	\label{lem:bounded} Let $X$ be the OU branching system starting from $x$. There exists a constant $C$ and a random variable $T$ such that with probability one
	\[ \forall_{t>T} \forall_{i\in \cbr{1,2,\ldots, |X_t|}} X_t(i) \in B(x e^{-\mu t} ,Ct ), \]
	where $B(0,r)$ denotes a ball of radius $r$ centred at $0$. 
\end{lemma}
\begin{proof}
	By the usual coupling argument it is enough to prove the fact for $x=0$. Let us denote by $A_n$ an event that $X$ up to time $n$ is contained in $B(0, C (n-1))$ for some constant $C$. By the Borel–Cantelli lemma to show our claim it is enough to prove $\sum_{n\geq 1} \pr{A_n'} <+\infty$. By the result from \cite{MR1664394}, the coupling construction and Gaussian concentration inequality we know that the supremum of the Ornstein-Uhlenbeck process on the interval $[0,n]$ can be stochastically dominated by $C_1 (\sqrt{\log(n)} + |G|)$, where $C_1>0$ is a certain constant and $G$ is a standard normal random variable. Let us fix any $\gamma > \lambda_p$ and estimate 
	\begin{align*}
		\pr{A_n'} &\leq \pr{A_n' \cap \cbr{|X_n|\leq e^{\gamma n}}} + \pr{|X_n|> e^{\gamma n}} \\
		&= \pr{\exists_{i\in\cbr{1,2,\ldots, n}} \sup_{s\in [0,n]} \norm{X_s(i)}{}{}> C(n-1)\cap \cbr{|X_n|\leq e^{\gamma n}} }+ \pr{|X_n|> e^{\gamma n}} \\
		&\leq \pr{C_1(\sqrt{\log(n)}+|G|) > C(n-1) }e^{\gamma n} + \pr{|X_n|> e^{\gamma n}}, 
	\end{align*}
	where $X_s(i)$ denotes the position of the particle $i$ or its ancestor at time $s\in [0,n]$. The proof is concluded by using the estimations of the Gaussian tail in the first term and Fact \ref{fact:aaa} together with the Chebyshev inequality for handling the second one. 
\end{proof}

We will also need 
\begin{fact}
	\label{fact:momentsSuperCritical} Let $\cbr{X_t}_{t\geq 0}$ be the OU branching system starting from $x$ and $\lambda_p > 2 \mu$. Moreover, let $f\in \pspace{}$. Then
	\[ \sup_{t>0} \evx{} \rbr{e^{-(\lambda_p - \mu) t} \ddp{X_t}{\tilde{f}}}^2 < +\infty. \]
	Moreover there exists $c:\R_+\mapsto \R_+$ such that $c(u)\rightarrow 0$ as $u\searrow 0$ and 
	\begin{equation}
		\limsup_{t\rightarrow +\infty} \evx{} \rbr{e^{-(\lambda_p - \mu) t} \rbr{\ddp{X_t}{\tilde{f}} -\ddp{X_t}{l_u}} }^2 < c(u), \label{eq:largeApproximation} 
	\end{equation}
	where $l_u$ is given by \eqref{eq:lu}. 
\end{fact}
\begin{proof}
	We recall \eqref{eq:moments} and \eqref{eq:moments-formula}. By Lemma \ref{lem:decay} one sees that $w'(x,t,0) \cleq (1+\norm{x}{}{n}) e^{(\lambda_p - \mu)t}$. Now we check that 
	\begin{multline*}
		w''(x,t,0) \cleq e^{\lambda_p t} \T{t} \tilde{f}^2(x) + \intc{t} e^{\lambda_p (t-s)}\T{t-s} \sbr{\rbr{(1+\norm{\cdot}{}{n}) e^{(\lambda_p - \mu)s}}^2}(x) \dd{s} \\
		\cleq e^{\lambda_p t} \T{t} \tilde{f}^2(x) + e^{\lambda_p t} \intc{t} e^{(\lambda_p - 2\mu) s}\T{t-s} \sbr{{(1+\norm{\cdot}{}{2n})}}(x) \dd{s} \cleq e^{2(\lambda_p - \mu) t} (1+\norm{x}{}{2n}). 
	\end{multline*}
	Inequality \eqref{eq:largeApproximation} follows by \eqref{eq:strongEstimation} and the above calculations. Indeed, one checks that 
	\begin{equation*}
		\limsup_{t\rightarrow +\infty } e^{-2(\lambda_p - \mu) t}w''(x,t,0) \leq c(u) (1+\norm{x}{}{n}), 
	\end{equation*}
	for some function $c(u)$ such that $c(u)\rightarrow 0$ as $u\searrow 0$. 
\end{proof}

\paragraph{Proof of Theorem \ref{thm:clt2}.} Our first aim will be to prove the convergence of the spatial fluctuations. To this end we denote
\[ Y_1(t) := e^{-(\lambda_p-\mu)t} \rbr{\ddp{X_t}{f} - |X_t| \ddp{f}{\eq} } = e^{-(\lambda_p-\mu)t} \ddp{X_t}{\tilde{f}}, \]
where $\tilde{f}(x) = f(x) - \ddp{f}{\eq}$. We have 
\begin{multline*}
	Y_1(t+s) = e^{-(\lambda_p - \mu) (t+s)} \sum_{i=1}^{|X_{t+s}|} \tilde{f}(X_{t+s}(i)) = e^{-(\lambda_p - \mu) t} \sum_{i=1}^{|X_{t}|} e^{-(\lambda_p - \mu) s} \sum_{j=1}^{|X_s^{i,t}|}\tilde{f}(X_{s}^{i,t}(j)) \\
	= e^{-(\lambda_p - \mu) t} \sum_{i=1}^{|X_{t}|} e^{-(\lambda_p - \mu) s} \rbr{\sum_{j=1}^{|X_s^{i,t}|} (\tilde{f}(X_{s}^{i,t}(j)) - \tilde{f}(\tilde{X}_{s}^{i,t}(j))) + \tilde{f}(\tilde{X}_{s}^{i,t}(j))}, 
\end{multline*}
where $\{X^{i,t}_s\}_s$ denotes the subsystem originating from the particle $X_t(i)$ and $\{\tilde{X}^{i,t}_s\}_s$ is a coupled version of $X^{i,t}$ starting from $0$. We write 
\begin{equation}
	Y_1(t+s) = e^{-(\lambda_p - \mu) t} \sum_{i=1}^{|X_{t}|} e^{-(\lambda_p - \mu) s} \sum_{j=1}^{|X_s^{i,t}|} (\tilde{f}(X_{s}^{i,t}(j)) - \tilde{f}(\tilde{X}_{s}^{i,t}(j))) + e^{-(\lambda_p - \mu) (t+s)} \sum_{i=1}^{|X_{t}|} \ddp{ \tilde{X}_{s}^{i,t}}{\tilde{f}}. \label{eq:motyka} 
\end{equation}
Let us first deal with the second term
\[ Y_2(t+s) := e^{-(\lambda_p - \mu) (t+s)} \sum_{i=1}^{|X_{t}|} \ddp{ \tilde{X}_{s}^{i,t}}{\tilde{f}_s} + e^{- (\lambda_p - \mu)(t+s)} l_s |X_{t+s}|, \]
where $\tilde{f}_s(x) := f(x) - \T{s}f(0)$ and $l_s := \T{s}f(0) - \ddp{f}{\eq}$. By Lemma \ref{lem:decay} we know that $|l_s| \cleq e^{-2\mu s}$ and therefore
\[ \evx{} |e^{- (\lambda_p - \mu)(t+s)} l_s |X_{t+s}|| \cleq e^{\mu(t+s) - 2\mu s}. \]
From now on, whenever we prove convergence, we will assume that $s=2 t$. With this convention the above expression converges to $0$. We denote the first summand of $Y_2$ by $Y_3$. As we explained in the proof of Theorem \ref{thm:clt1} the systems $\tilde{X}^{i,t}$ are i.i.d. Moreover their only connection with $X_t$ is via the number of particles $|X_t|$. Below, we will use the conditional expectation given $X_t$ which will be denoted as $\ev{}_{X_t}$. To ease the notation we will write simply $\ev{}_0 \tilde{X}^{i,t}$ instead of $\ev{}_{X_t} \tilde{X}^{i,t}$. Using the conditional expectation and the particular choice of $\tilde{f}_s$ we calculate 
\begin{align*}
	\evx{}\rbr{Y_3(t+s)}^2 &= e^{-2(\lambda_p - \mu) (t+s)} \evx{} \sum_{i=1}^{|X_{t}|} \sum_{j=1}^{|X_{t}|}\ev{}_{X_t} \rbr{\ddp{ \tilde{X}_{s}^{i,t}}{\tilde{f}_s}\ddp{ \tilde{X}_{s}^{j,t}}{\tilde{f}_s}}\\
	&=e^{-2(\lambda_p - \mu) (t+s)} \evx{} \sum_{i=1}^{|X_{t}|} \ev{}_0 \rbr{\ddp{ \tilde{X}_{s}^{i,t}}{\tilde{f}_s}^2}. 
\end{align*}
By Fact \ref{fact:momentsSuperCritical} one can prove that $\ev{}_0 \rbr{\ddp{ \tilde{X}_{s}^{i,t}}{\tilde{f}_s}^2} \cleq e^{2(\lambda_p - \mu) s}$ therefore
\[ \evx{}\rbr{Y_3(t+s)}^2 \cleq e^{(-\lambda_p + 2\mu) t} \rightarrow 0. \]
In this way we proved that $Y_2(s+t)\rightarrow 0$ in probability. By arguments in Section \ref{sec:approximation} we may assume that $f$ has the second derivative which is bounded by a polynomial. To this end we check that \eqref{eq:assumptionUniform} follows by \eqref{eq:largeApproximation}.

Now we decompose the first term of \eqref{eq:motyka} according to the formula 
\begin{align}
	&Y_5(t+s) + Y_6(t+s) := \nonumber\\
	&e^{-(\lambda_p - \mu) (t+s)} \sum_{i=1}^{|X_{t}|} \sum_{j=1}^{|X_s^{i,t}|} E(i,t,s) +e^{-(\lambda_p - \mu) (t+s)} \sum_{i=1}^{|X_{t}|} \sum_{j=1}^{|X_s^{i,t}|} \grad f(\tilde{X}_{s}^{i,t}(j)) \circ ({X}_{s}^{i,t}(j)-\tilde{X}_{s}^{i,t}(j)),\label{eq:tobeusedagain} 
\end{align}
where $E(i,t,s) := \tilde{f}(X_{s}^{i,t}(j)) - \tilde{f}(\tilde{X}_{s}^{i,t}(j)) - \grad f(\tilde{X}_{s}^{i,t}(j)) \circ ({X}_{s}^{i,t}(j)-\tilde{X}_{s}^{i,t}(j))$. We now notice that by the coupling properties we have 
\begin{equation}
	{X}_{s}^{i,t}(j)-\tilde{X}_{s}^{i,t}(j) = {X}_{t}(i) e^{-\mu s}. \label{eq:tmpCoupling} 
\end{equation}
We have $|E(i,t,s)|\cleq R^n ( \norm{{X}_{t}(i)}{}{} e^{-\mu s})^2$ if only ${X}_{s}^{i,t}(j),\tilde{X}_{s}^{i,t}(j)$ belong to the ball of radius $R$. Let us now denote the event that this condition holds for all $i,j$ with $R= C (t+s) + \norm{x}{}{} e^{-\mu t}$ by $A_t$. By Fact \ref{fact:aaa} we thus prove
\[ \evx{} |Y_5(t+s)| 1_{A_t} \cleq e^{\mu(s+t)} ((t+s) + e^{-\mu t} \norm{x}{}{} )^{3n} e^{-2\mu s} \rightarrow 0. \]
By Lemma \ref{lem:bounded} we have $\pr{A_t} \rightarrow 1$, hence we proved that $Y_5(t+s)\rightarrow 0$ in probability. Using again \eqref{eq:tmpCoupling} we may rewrite $Y_6$ as 
\begin{multline*}
	Y_6(t+s) = e^{-(\lambda_p - \mu) t} \sum_{i=1}^{|X_{t}|} X_t(i) \circ \rbr{ Z_s^i - z_s } + z_s \circ \rbr{e^{-(\lambda_p - \mu) t}\sum_{i=1}^{|X_{t}|} X_t(i)} \\
	+ \ddp{\grad f}{\eq}\circ \rbr{e^{-(\lambda_p - \mu) t} \sum_{i=1}^{|X_{t}|} X_t(i) (e^{-\lambda_p s}|X^{i,t}_s|-1)} + \ddp{\grad f}{\eq}\circ \rbr{e^{-(\lambda_p - \mu) t} \sum_{i=1}^{|X_{t}|} X_t(i) } \\
	=: Y_7(t+s) + Y_8(t+s) + Y_9(t+s) + Y_{10}(t+s), 
\end{multline*}
where $h(x) = \grad f(x) - \ddp{\grad f}{\eq} $, $Z^i_s := e^{-\lambda_p s} \ddp{\tilde{X}^{i,t}_s}{h}$ and $z_s:= \ev{}_0Z^i_s$ (this does not depend on $i$ since $Z^i_s$ are i.i.d.). Simple calculations using \eqref{eq:moments-formula} and Lemma \ref{lem:decay} (second estimation in \eqref{eq:decay1}) reveal that $\norm{z_s}{}{} \cleq e^{-2\mu s}$. Moreover by Fact \ref{fact:momentsSuperCritical} the covariance matrix $\cov_0(Z^i_s)$ is bounded by some $C$ (in a sense that each entry is bounded). Using conditioning with respect to $|X_t|$ we have (we skip the index )
\[ \evx{} Y_7(t+s)^2 = {e^{-2(\lambda_p - \mu) t} \evx{} \sum_{i=1}^{|X_{t}|} X_t(i)^T \cov_0{(Z^i_s)}X_t(i)} \cleq {e^{-2(\lambda_p - \mu) t} \evx{} \sum_{i=1}^{|X_{t}|} \norm{X_t(i)}{}{2}} \rightarrow 0, \]
where $^T$ denotes the transposition and the convergence holds by assumption that $\lambda_p>2\mu$. The convergence: $\evx{}(Y_9(t+s))^2 \rightarrow 0$ follows in a very similar fashion and is left to the reader. Now by Fact \ref{fact:martingaleConvergence} one easily checks that $Y_8(t+s)\rightarrow 0 $ a.s and $Y_{10}$ converges to the same limit as in the thesis of Theorem \ref{thm:clt2}. A revision of the steps above reveals that the limit (in probability) of $Y_1$ is the same as the limit of $Y_{10}$. This in fact is also enough to conclude the second convergence in Theorem \ref{thm:clt2}. To end the proof we notice that to obtain the joint convergence \eqref{eq:triple} one may use the same methods as in the proof of Theorem \ref{thm:clt1}.

\subsection{CLT for critical branching rate}
\begin{fact}
	\label{fact:momentsCritical} Let $\cbr{X_t}_{t\geq 0}$ be the OU branching system and $\lambda_p = 2 \mu$. Moreover, let $f\in \pspace{}$, then there exists a constant $C$ such that
	\[ \evx{} \rbr{e^{-(\lambda_p/2) t} t^{-1/2} \ddp{X_t}{\tilde{f}}} \rightarrow 0 \quad \text{ as } t\rightarrow +\infty. \]
	\begin{equation}
		\evx{} \rbr{e^{-(\lambda_p/2) t} t^{-1/2} \ddp{X_t}{\tilde{f}}}^2 \rightarrow \sigma_f^2, \quad \var_x\rbr{e^{-(\lambda_p/2) t} t^{-1/2} \ddp{X_t}{\tilde{f}}} \rightarrow \sigma_f^2, \label{eq:criticalSecondMoments} 
	\end{equation}
	where $\sigma_f^2$ is the same as in \eqref{eq:sigmafCritical}. Moreover, $\sup_{t\geq \delta } \evx{} \rbr{e^{-(\lambda_p/2) t} t^{-1/2} \ddp{X_t}{\tilde{f}}}^4 <+\infty$ for any $\delta>0$. 
\end{fact}
\begin{proof}
	The first convergence follows easily by \eqref{eq:decay1} in Lemma \ref{lem:decay}. Using \eqref{eq:moments-formula} and \eqref{eq:moments} again we calculate the second moment 
	\begin{multline}
		\evx{} \rbr{e^{-(\lambda_p/2)t}t^{-1/2}\ddp{X_t}{\tilde{f}}}^2 = t^{-1}\T{t} \tilde{f}^2(x) + 2\lambda p e^{-\lambda_p t} t^{-1} \intc{t} e^{\lambda_p (t-s)} \T{t-s} \sbr{\rbr{ e^{\lambda_p s}\T{s} \tilde{f}(\cdot)}^2}(x) \dd{s} \\
		= t^{-1}\T{t} \tilde{f}^2(x) + 2\lambda p t^{-1} \intc{t} \T{t-s} \sbr{\rbr{ e^{(\lambda_p/2) s}\T{s} \tilde{f}(\cdot)}^2}(x) \dd{s}. \label{eq:NnwSecondMoment} 
	\end{multline}
	
	We recall that $\lambda_p/2 = \mu$. Using Lemma \ref{lem:decay} (equation \eqref{eq:decay2}) and elementary considerations, we obtain that the limit of the above expression is the same as the one of 
	\begin{multline*}
		2\lambda p t^{-1} \intc{t} \T{t-s} \sbr{\rbr{\sum_{i=1}^d x_i \ddp{\frac{
		\partial f}{
		\partial x_i} }{\eq} }^2}(x) \dd{s} = \\
		2\lambda p \sum_{i=1}^d \rbr{\ddp{\frac{
		\partial f}{
		\partial x_i} }{\eq}^2 t^{-1}\intc{t} \T{t-s} \sbr{{ x_i^2 }}(x) \dd{s}} + 2\lambda p \sum_{i\neq j} \rbr{\ddp{\frac{
		\partial f}{
		\partial x_i} }{\eq} \ddp{\frac{
		\partial f}{
		\partial x_j} }{\eq} t^{-1} \intc{t} \T{t-s} \sbr{{ x_i x_j }}(x) \dd{s}}. 
	\end{multline*}
	One easily checks that $\T{t}[x_ix_j](x) \rightarrow 0$ hence the second term disappears in the limit. We also have $\T{t}[x_i^2](x) \rightarrow \frac{\sigma^2}{2 \mu}$ and so the whole expression converges to $\sigma_f$ given by \eqref{eq:sigmafCritical}.
	
	Obviously the limit of variances is the same. We also conclude that for any $t\geq 0$
	\[ w''(x,t,0) \cleq (1+ \norm{x}{}{2n}) e^{\lambda_p t} t. \]
	Similarly we investigate $w'''(x,t,0)$. By \eqref{eq:moments-formula} we have
	\[ |w'''(x,t,0)| \cleq e^{\lambda_p t}\T{t} { |\tilde{f}|^3(x) } + \left.\intc{t} e^{\lambda_p (t-s)} \T{t-s} \sbr{ w''(\cdot, s,0) w'(\cdot, s,0)}(x) \dd{s} .\right. \]
	Using the fact that by (\ref{eq:decay1}) $|w'(x,t,0)|\cleq e^{(\lambda_p -\mu)t}(1+\|x\|^n)$, together with the above estimate on $w''$ and the fact that $\T{t}\sbr{(1+\norm{\cdot}{}{n})^3}(x)\cleq (1+\norm{x}{}{3n})$, we get
	\[ |w'''(x,t,0)| \cleq (1+\norm{x}{}{3n})e^{\lambda_p t} + e^{\lambda_p t} \left.\intc{t} e^{(\lambda_p/2) s} s \T{t-s} \sbr{ 1+ \norm{\cdot}{}{3n} }(x) \dd{s} \right. \cleq e^{((3/2) \lambda_p)t}t (1+ \norm{x}{}{3n}). \]
	Finally, we will also need the fourth moment. By \eqref{eq:moments-formula} and the estimates above we get 
	\begin{multline*}
		\evx{} \rbr{e^{-(\lambda_p/2)t}t^{-1/2}\ddp{X_t}{\tilde{f}}}^4 \cleq e^{-\lambda_p t}t^{-2}\T{t} { \tilde{f}^4(x) } \\+ e^{-2\lambda_p t} t^{-2} \intc{t} e^{\lambda_p (t-s)} \T{t-s} \sbr{ w''(\cdot,s,0)^2 + w'''(\cdot,s,0) w'(\cdot,s,0) }(x) \dd{s} \\
		\cleq e^{-\lambda_p t} t^{-2} (1+\norm{x}{}{4n}) + e^{-\lambda_p t} t^{-2}\intc{t} e^{\lambda_p s}s^{2} \T{t-s} \sbr{ (1 + \norm{\cdot}{}{4n}) }(x) \dd{s} \cleq (1 + \norm{x}{}{4n}). 
	\end{multline*}
	It is now easy to check that for $t > \delta$, 
	\begin{equation*}
		\evx{} \rbr{e^{-(\lambda_p/2)t} t^{-1/2}\rbr{\ddp{X_t}{\tilde{f}} - \evx{}\ddp{X_t}{\tilde{f}}} }^4 \cleq (1 + \norm{x}{}{4n}). 
	\end{equation*}
\end{proof}
We have yet to prove \eqref{eq:sigmafCriticalHermite}. By \eqref{eq:equilibrium} we have $\frac{
\partial }{
\partial x_i } \eq(x) = -2(\mu/\sigma^2) x_i \eq(x)$. Therefore,
\[ \ddp{\frac{
\partial f}{
\partial x_1} }{\eq} = \ddp{f}{ \frac{
\partial \eq}{
\partial x_1} } = \ddp{f}{ -2(\mu/\sigma^2) x_1 \eq } = -2(\mu/\sigma^2) f_{1,0,\ldots,0}. \]
Other coordinates can be treated in the same way, giving (\ref{eq:sigmafCriticalHermite}).

Now we are ready for 
\paragraph{Proof of Theorem \ref{thm:cltCritical}.} In this proof we will us both ideas of the proof of Theorem \ref{thm:clt1} and Theorem \ref{thm:clt2}. We start with the following random vector
\[ Z_1(t):=\rbr{e^{-\lambda_p t} |X_t|, e^{-(\lambda_p/2)t} (|X_t| - e^{\lambda_p t} V_\infty), e^{-(\lambda_p/2)t}t^{-1/2} \ddp{X_t}{\tilde{f}} }. \]
Let $n\in \mathbb{N}$ be fixed later and let us write
\[ Z_1(n t):=\rbr{e^{- n \lambda_p t} \ddp{X_{nt}}{1}, e^{-(n\lambda_p/2)t} (|X_{nt}| - e^{n\lambda_p t} V_\infty) , e^{-(n\lambda_p/2)t}(nt)^{-1/2} \sum_{i=1}^{|X_t|} \ddp{X^{i,t}_{(n-1)t}}{\tilde{f}} }, \]
where $\cbr{X^{i,s}_t}_t$ denotes the subsystem originating from the particle $X_s(i)$. Analogously in the proof of Theorem \ref{thm:clt1} the second term is equal to $ e^{-(n\lambda_p/2)t}\sum_{i=1}^{|X_{nt}|} \rbr{1- V_\infty^i}$, where $V^i_\infty$ are independent copies of $V_\infty$ arising from $i$-th particle. Next, we couple each $X^{i,t}$ with the branching system starting from one particle located at $0$. To this end we use the same methods as in Fact \ref{fact:coupling} for particles movements and retain the branching structure. The coupled system is denoted by $\tilde{X}^{i,t}$. We write
\[ H_t^n:=e^{-(n\lambda_p/2)t}(nt)^{-1/2} \rbr{\sum_{i=1}^{|X_t|} \ddp{X^{i,t}_{(n-1)t}}{\tilde{f}} - \sum_{i=1}^{|X_t|} \ddp{\tilde{X}^{i,t}_{(n-1)t}}{\tilde{f}}} \]
By arguments in Section \ref{sec:approximation} we may assume that $f$ has the second derivative which is bounded by a polynomial. To this end we check that \eqref{eq:assumptionUniform} follows by \eqref{eq:criticalSecondMoments} (and some elementary calculations taking into account the form of the limiting variance). This expression is harder to analyse compared to the case of the small branching rate. Now we will use the methods and notation of the proof of Theorem \ref{thm:clt2}. Recalling that $\lambda_p = 2\mu$ one notices that $H_t^n$ is the same expression as $ (nt)^{-1/2} \rbr{Y_5(t+(n-1)t)+Y_6(t+(n-1)t)}$ when one puts $s=(n-1)t$ (see (\ref{eq:tobeusedagain})). $Y_5$ can be handled in the same way, if only $n>2$. Next, using the coupling properties we decompose $Y_6$ in the following way 
\begin{multline}
	(n t)^{-1/2}Y_6(n t) = (n t)^{-1/2}e^{-(\lambda_p /2) t} \sum_{i=1}^{|X_{t}|} X_t(i) \circ K^i_s =(n t)^{-1/2} e^{-(\lambda_p/2) t} \sum_{i=1}^{|X_{t}|} X_t(i) \circ (K^i_s-k_s) \\+ k_s\circ \rbr{(n t)^{-1/2} e^{-(\lambda_p/2) t} \sum_{i=1}^{|X_{t}|} X_t(i) } ,\label{eq:tmp83} 
\end{multline}
where this time $K^i_s :=e^{-\lambda_p s} \ddp{\grad f}{\tilde{X}_{s}^{i,t}} $ and $k_s := \ev{}_0 K^i_s$ (this is independent of $i$ as $K^i_s$ are i.i.d). $\ev{}_0\norm{K_s^i-k_s}{}{2}$ are uniformly bounded in $s$ which follows in the same way as estimation of $e_t$ in \eqref{eq:LLNL2Convergence}. Using conditioning in a similar manner as in the estimation of $Y_7$ in the proof of Theorem \ref{thm:clt2} we obtain 
\begin{multline*}
	\evx{}\rbr{(n t)^{-1/2} e^{-(\lambda_p/2) t} \sum_{i=1}^{|X_{t}|} X_t(i) \circ (K^i_s-k_s)}^2 = (n t)^{-1}{e^{-\lambda_p t} \evx{} \sum_{i=1}^{|X_{t}|} X_t(i)^T \cov_0{(K^i_s)}X_t(i)}\\
	\cleq (n t)^{-1}{e^{-\lambda_p t} \evx{} \sum_{i=1}^{|X_{t}|} \norm{X_t(i)}{}{2}} \cleq (n t)^{-1} (1+\norm{x}{}{n}) \rightarrow 0, 
\end{multline*}
where again we used \eqref{eq:moments-formula} and \eqref{eq:moments}. Further, we notice also that $\ddp{x}{\eq} = 0$ therefore by Fact \ref{fact:momentsCritical} and the fact that $k_s$ is bounded we know that the second moment of the second term of \eqref{eq:tmp83} is bounded by $c/(2\sqrt{n})$, where $c>0$ is a certain constant. Thus by \eqref{eq:weakL2estimation}, we may conclude that for any $n$ there exists $t_n$ such that for any $t\geq t_n$, 
\begin{align}
	\label{eq:H_tn} \|H_t^n\|_2 \le 2c/\sqrt{n}. 
\end{align}

We recall that $\tilde{X}^{i,t}$'s are i.i.d. branching particle systems. Let us denote 
\begin{equation}
	Z_t^i := e^{-((n-1)\lambda_p/2) t }(n t)^{-1/2} \ddp{\tilde{X}^{i,t}_{(n-1)t}}{\tilde{f}} = \rbr{\frac{n-1}{n}}^{1/2}\rbr{ e^{-((n-1)\lambda_p/2) t }((n-1) t)^{-1/2} \ddp{\tilde{X}^{i,t}_{(n-1)t}}{\tilde{f}}} \label{eq:newZZ} 
\end{equation}
and $z_t := \ev{}_0 Z_t^i$. By Lemma \ref{lem:decay} one checks that $|z_t| = e^{((n-1)\lambda_p/2) t }(n t)^{-1/2} |\T{t} \tilde{f}(0)|\cleq e^{((n-1)(\lambda_p/2 - 2\mu)) t }$. Therefore for $n>2$ we have
\[ \ev{} e^{-(\lambda_p/2)t} \sum_{i=1}^{|X_t|} |z^i_t| = e^{(\lambda_p/2)t} |\ev{}_0 Z_t^1| \cleq e^{(\lambda_p/2)t}e^{((n-1)(\lambda_p/2 - 2\mu)) t } \rightarrow 0, \quad \text{as } t\rightarrow +\infty. \]
Using the facts above we conclude that for any fixed $n> 2$, the expression $Z_1$ is asymptotically equivalent to
\[ Z^n_2(t) = \rbr{e^{-\lambda_p t} |X_t|, e^{-(n\lambda_p/2)t}\sum_{i=1}^{|X_{nt}|} \rbr{1- V_\infty^i}, H_t^n + e^{-(\lambda_p/2)t} \sum_{i=1}^{|X_t|} (Z^i_t - z^i_t) }. \]
This expression differs from the analogous in the proof of Theorem \ref{thm:clt1} only by additional term $H_t^n$. Let us skip it for a moment and write
\[ \tilde{Z}^n_2(t) = \rbr{e^{-\lambda_p t} |X_t|, e^{-(n\lambda_p/2)t} \sum_{i=1}^{|X_{nt}|} \rbr{1- V_\infty^i}, e^{-(\lambda_p/2)t} \sum_{i=1}^{|X_t|} (Z^i_t - z^i_t) }. \]
Returning to the flow of the proof of Theorem \ref{thm:clt1} one can prove that
\[ \tilde{Z}^n_2(t) \rightarrow^d \rbr{W, \sqrt{W} G_1, \rbr{\frac{n-1}{n}}^{1/2} \sqrt{W} G_2}, \]
conditionally on $Ext^c$. In order to achieve it one simply uses Fact \ref{fact:momentsCritical} instead of Fact \ref{fact:momentsSubcritical}. The additional term $\rbr{\frac{n-1}{n}}^{1/2}$ originates from the same term in \eqref{eq:newZZ}.

To finish the proof let us denote the law of the triple in the last expression by $\mathcal{L}_n$ and the one of $(W, \sqrt{W}G_1, \sqrt{W}G_2)$ by $\mathcal{L}_\infty$. For any $\epsilon>0$, there exists $n$ such that $m(\mathcal{L}_n,\mathcal{L}_\infty) \leq \epsilon$ and $c/\sqrt{n} \le \varepsilon$, where $c$ is the constant in \eqref{eq:H_tn}. Now we choose $T$ large enough to have $m(\mathcal{L}(\tilde{Z}^n_2(t)), \mathcal{L}_n) \leq \epsilon$, $m(\mathcal{L}({Z}^n_2(t)), \tilde{Z}^n_2(t))) \leq \epsilon$ and $m(\mathcal{L}({Z}^n_2(t)), \mathcal{L}({Z}_1(nt))) \leq \epsilon$ for any $t\geq T$. By applying the triangle inequality we get $m(\mathcal{L}, \mathcal{L}({Z}_1(nt))) \leq 4\epsilon$ for any $t\geq T$. The proof is concluded by obvious transformations. 
\section*{Appendix} 

\label{sec:appendix}
\begin{proof}
	[Proof of Fact \ref{fact:aaa}] Using \eqref{eq:laplace-aux} with $f(x) = 1 $ we get (we drop argument $x$)
	\[ \frac{d}{dt} w(t,\theta) = (\lambda p) w(t,\theta)^2-\lambda w(t,\theta) + \lambda (1-p),\quad w(0,\theta) = e^{-\theta}. \]
	The solution of this equation is (see e.g. \cite[Section III.5]{Athreya:2004xr})
	\[ w(t,\theta) = \frac{\lambda(1-p)(e^{-\theta} -1) - e^{-\lambda_p t}(\lambda p e^{-\theta} - \lambda(1-p)) }{\lambda p (e^{-\theta} -1) - e^{-\lambda_p t}(\lambda p e^{-\theta} - \lambda(1-p)) }. \]
	Now we want to investigate convergence of $e^{-\lambda_p t}|X_t|$. Its Laplace transform is 
	\begin{equation}
		L(t,\theta):=\frac{\lambda(1-p)(e^{-\theta e^{-\lambda_p t}} -1) - e^{-\lambda_p t}(\lambda p e^{-\theta e^{-\lambda_p t}} - \lambda(1-p)) }{\lambda p (e^{-\theta e^{-\lambda_p t}} -1) - e^{-\lambda_p t}(\lambda p e^{-\theta e^{-\lambda_p t}} - \lambda(1-p)) }. \label{eq:laplaceNumber} 
	\end{equation}
	Using the first order Taylor expansion and dropping terms of lower order we get 
	\begin{align*}
		L(t,\theta) &= \frac{\lambda(1-p)\rbr{ -\theta e^{-\lambda_p t} + o(\theta e^{-\lambda_p t}) } - e^{-\lambda_p t} \rbr{ -\lambda p \theta e^{-\lambda_p t} + o(e^{-\lambda_p t}) +\lambda_p } }{\lambda p \rbr{ -\theta e^{-\lambda_p t} + o(\theta e^{-\lambda_p t}) } - e^{-\lambda_p t}\rbr{ -\lambda p \theta e^{-\lambda_p t} + o(e^{-\lambda_p t}) +\lambda_p } }\\
		& \approx \frac{-\lambda(1-p)\theta e^{-\lambda_p t} - \lambda_p e^{-\lambda_p t} }{-\lambda p \theta e^{-\lambda_p t} -\lambda_p e^{-\lambda_p t} }. 
	\end{align*}
	Therefore
	\[ L(t,\theta) \rightarrow \frac{\lambda(1-p)\theta + \lambda_p }{\lambda p \theta +\lambda_p } = \frac{\theta-p (-2+\theta ) -1}{+p (2+\theta )-1} =:L(\theta) \:\:\text{ as }\:\: t\rightarrow +\infty. \]
	Taking $\theta \rightarrow +\infty$ it is easy to check that $\pr{V_\infty = 0 } = p_e$ (we recall that $p_e=\frac{1-p}{p}$), therefore $V_\infty>0$ on the set of non-extinction $Ext^c$. Let us now calculate the law of $V_\infty$ on the set of non-extinction
	\[ L(\theta) = \ev{} e^{-\theta V_\infty} = \ev{} e^{-\theta V_\infty} 1_{Ext} + \ev{} e^{-\theta V_\infty} 1_{Ext^c} = p_e + (1-p_e) \ev{} \rbr{e^{-\theta V_\infty}| Ext^c}. \]
	Therefore
	
	\[ \ev{} \rbr{e^{-\theta V_\infty}| Ext^c} = \frac{ \lambda_p }{\lambda p \theta +\lambda_p } = \frac{ (2p-1) }{p \theta + (2p-1) }. \]
	Further
	\[ \ev{}|X_t|^4 = \frac{
	\partial^4 w(t,0)}{
	\partial \theta^4} = \frac{e^{t \lambda _p} \left(-1+2 \left(-4+7 e^{t \lambda _p}\right) p+\left(8+16 e^{t \lambda _p}-36 e^{2 t \lambda _p}\right) p^2+8 e^{t \lambda _p} \left(-2+3 e^{2 t \lambda _p}\right) p^3\right)}{(-1+2 p)^3}. \]
	Now the second part of the fact follows. In order to prove that all moments are finite we notice that derivatives of \eqref{eq:laplaceNumber} are of the form
	\[ \frac{\dd{^n} }{\dd{\theta^n} } L(t,\theta) = \frac{ l(t,\theta) }{ (\lambda p (e^{-\theta} -1) - e^{-\lambda_p t}(\lambda p e^{-\theta} - \lambda(1-p)))^{2n} }, \]
	where $l(t,\theta)$ is a certain expression. Obviously the denominator is finite for $\theta = 0$ hence the proof is concluded by the properties of the Laplace transform (e.g. \cite[Chapter XIII.2]{Feller:1971cr}). 
\end{proof}

\subsubsection*{Acknowledgments} We would like to thank Dr Simon Harris for introducing to the topic and drawing our attention to expansion described in Remark \ref{rem:martingales}. 
\bibliographystyle{abbrv} 
\bibliography{branching} \end{document}